\documentclass{elsarticle}

\usepackage{amsmath,amsthm,amssymb}
\usepackage{epsfig}
\usepackage{amssymb}
\usepackage{amsmath}
\usepackage{amssymb}
\usepackage{amsmath,amsthm}
\usepackage[latin1]{inputenc}
\usepackage[T1]{fontenc}
\usepackage{path}
\usepackage{ae,aecompl}
\usepackage{amsfonts}
\usepackage{amsxtra}
\usepackage{bbm,euscript,mathrsfs}

\usepackage{enumitem}
\setenumerate{label={\rm (\alph{*})}}

\usepackage[frak,hyperref]{paper_diening}

\numberwithin{equation}{section}

\newcommand{\dy}{\,\mathrm{d}y}
\newcommand{\dx}{\,\mathrm{d}x}
\newcommand{\dt}{\,\mathrm{d}t}
\newcommand{\dxt}{\,\mathrm{d}x\,\mathrm{d}t}
\newcommand{\ds}{\,\mathrm{d}\sigma}

\newcommand{\R}{\mathbb{R}}
\renewcommand{\div}{\operatorname{div}}

\newcommand{\N}{{\mathbb{N}}}

\newcommand{\ep}{\bfvarepsilon}

\newcommand{\bfell}{\boldsymbol{\ell}}

\providecommand{\Qz}{\ensuremath{Q_0}}
\providecommand{\Bz}{\ensuremath{B_0}}
\providecommand{\Iz}{\ensuremath{I_0}}

\providecommand{\Rdr}{{\setR^3}}

\providecommand{\Ma}{\ensuremath{\mathcal{M}}}
\providecommand{\Mas}{\ensuremath{\mathcal{M}_\sigma}}

\providecommand{\Oal}{\ensuremath{\mathcal{O}_\lambda}}

\providecommand{\zal}{\ensuremath{\bfz^\alpha_\lambda}}

\def\IntBRx0{\int_{B_R(x_0)}}

\def\IntB2R{\int_{B_{2R}}}
\def\IntBtRx0{\int_{B_{tR}(x_0)}}

\def\Xint#1{\mathchoice%
    {\XXint\displaystyle\textstyle{#1}}%
    {\XXint\textstyle\scriptstyle{#1}}%
    {\XXint\scriptstyle\scriptscriptstyle{#1}}%
    {\XXint\scriptscriptstyle\scriptscriptstyle{#1}}%
    \!\int}
\def\XXint#1#2#3{{\setbox0=\hbox{$#1{#2#3}{\int}$}%
    \vcenter{\hbox{$#2#3$}}\kern-.6\wd0}}
\def\fint{\Xint-}

\def\BBint#1#2#3{{\setbox0=\hbox{$#1{#2#3}{\int}$}%
    \vcenter{\hbox{$#2#3$}}\kern-.8\wd0}}

\def\CCint#1#2#3{{\setbox0=\hbox{$#1{#2#3}{\int}$}%
    \vcenter{\hbox{$#2#3$}}\kern-.9\wd0}}

\def\FFint#1#2#3{{\setbox0=\hbox{$#1{#2#3}{\int}$}
    \vcenter{\hbox{$#2#3$}}\kern-1.0\wd0}}

\def\WWint#1#2#3{{\setbox0=\hbox{$#1{#2#3}{\int}$}%
    \vcenter{\hbox{$#2#3$}}\kern-1.1\wd0}}

\def\Gfffffint#1#2#3{{\setbox0=\hbox{$#1{#2#3}{\int}$}%
    \vcenter{\hbox{$#2#3$}}\kern-1.3\wd0}}

\newcommand{\newmathop}[2]{\newcommand{#1}{\mathop{\mit{#2}}}}
\newmathop{\Minttext}{\int\limits\hspace{ -0.92em}-}
\newmathop{\Mint}{\int\hspace{ -1.075em}-}

\allowdisplaybreaks

\def\BBint#1#2#3{{\setbox0=\hbox{$#1{#2#3}{\int}$}%
    \vcenter{\hbox{$#2#3$}}\kern-.8\wd0}}

\def\CCint#1#2#3{{\setbox0=\hbox{$#1{#2#3}{\int}$}%
    \vcenter{\hbox{$#2#3$}}\kern-.9\wd0}}

\def\FFint#1#2#3{{\setbox0=\hbox{$#1{#2#3}{\int}$}
    \vcenter{\hbox{$#2#3$}}\kern-1.0\wd0}}

\def\WWint#1#2#3{{\setbox0=\hbox{$#1{#2#3}{\int}$}%
    \vcenter{\hbox{$#2#3$}}\kern-1.1\wd0}}

\def\Gfffffint#1#2#3{{\setbox0=\hbox{$#1{#2#3}{\int}$}%
    \vcenter{\hbox{$#2#3$}}\kern-1.3\wd0}}

\DeclareMathOperator{\Div}{div}

\DeclareMathOperator{\LD}{LD}

\newcommand{\s}{\mathbb S}

\newcounter{formel}

\newtheorem{defs}{Definition}[section]
\newtheorem{them}[defs]{THEOREM}

\newtheorem{lems}[defs]{Lemma}
\newtheorem{lemma}[defs]{Lemma}

\newtheorem{corollary}[defs]{Corollary}

\newtheorem{remark}[defs]{Remark}
\newtheorem{props}[defs]{Proposition}

\newcommand{\Bog}{\ensuremath{\text{\rm Bog}}}

\providecommand{\dtau}{\,\mathrm{d}\tau}

\begin{document}

\begin{frontmatter}

  \title{The $\mathcal A$-Stokes approximation for non-stationary problems}

  \author{Dominic Breit}
  \ead{d.breit@hw.ac.uk}

  \address{School of Mathematical \& Computer Sciences, Heriot-Watt University,
Riccarton Edinburgh, EH14 4AS UK}

  \begin{abstract}
Let $\mathcal A$ be an elliptic tensor.  A function $\bfv\in L^1(I;\LD_{\Div}(B))$ is a solution to the non-stationary $\mathcal A $-Stokes problem iff
\begin{align}\label{abs}
\int_Q \bfv\cdot\partial_t\bfphi\dxt-\int_Q \mathcal A(\ep(\bfv),\ep(\bfphi))\dx=0\quad\forall\bfphi\in C^{\infty}_{0,\Div}(Q),
\end{align}
where $Q:=I\times B$, $B\subset\R^d$ bounded. If the l.h.s. of \eqref{abs} is not zero but small we talk about almost solutions. We present an approximation result in the fashion of the $\mathcal A$-caloric approximation for the non-stationary $\mathcal A $-Stokes problem. Precisely, we show
that every almost solution $\bfv\in L^p(I;W^{1,p}_{\Div}(B))$, $1<p<\infty$, to (\ref{abs}) can be approximated by a solution to (\ref{abs}) in the $L^s(I;W^{1,s}(B))$-sense for all $s<p$. So, we extend the stationary $\mathcal A$-Stokes approximation from \cite{BrDF} to parabolic problems. 
  \end{abstract}

\end{frontmatter}

\section{Introduction}
\label{sec:intro}
A very crucial tool in the analysis of nonlinear partial differential equations is a comparison with solutions to linear problems. This powerful idea firstly appears in the work of De Giorgi \cite{DeG61} in the context of minimal surfaces. Roughly speaking we consider the nonlinear problem locally as a perturbation of a linear partial differential equation and try to transfer qualitative properties from the linear theory. Solutions to linear problems typically are smooth so we are hoping to approximate solutions to nonlinear problems by solutions to linear ones in an appropriate way in  order to establish regularity properties (of the solutions to nonlinear equations). \\
Let $\mathcal A:\R^{d\times D}\rightarrow\R^{d\times D}$ be an elliptic tensor, i.e.
\begin{align}\label{eq:ell}
\mathcal A(\bfxi,\bfxi):=\mathcal A\bfxi:\bfxi\geq\lambda|\bfxi|^2\quad\forall \bfxi\in\R^{d\times D}
\end{align}
with $\lambda>0$ and $B\subset \R^d$ open and bounded with Lipschitz boundary (for instance a ball).
We call a function $\bfu\in W^{1,1}(\Omega)$ with
\begin{align}
\int_B \mathcal A(\nabla\bfu,\nabla\bfphi)\dx=0\quad\forall\bfphi\in C^{\infty}_0(B)
\end{align}
$\mathcal A$-harmonic on $B$. On the other, if we have for some $\delta\ll1$
\begin{align}
\Big|\int_B \mathcal A(\nabla\bfu,\nabla\bfphi)\dx\Big|\leq \delta\quad \forall\bfphi\in C^{\infty}_0(B),
\end{align}
i.e. the integral $\int_B \mathcal A(\nabla\bfu,\nabla\bfphi)\dx$ is small (compared to some norms of $\bfu$ and $\bfphi$) we call $\bfu$ almost $\mathcal A$-harmonic. De Giorgi's observation in \cite{DeG61} was the fact that almost harmonic functions can be approximated by harmonic ones. Precisely, if $\bfu\in W^{1,2}(B)$ is almost harmonic it can be approximated with respect to the $L^2(\Omega)$-norm. Since this usually does not suffice to show regularity one needs in addition a Caccioppoli inequality. It bounds the $L^2(\Omega)$-norm of the gradient in terms of the $L^2(B)$-norm of the function itself. Both together - harmonic approximation and Caccioppoli-inequality -  finally yield local $C^{1,\alpha}$-estimates.\\
Since the pioneering work of De Giorgi a lot of improvement and generalizations have been done (see \cite{DuzMin09h} for an overview). For instance the $p$-harmonic approximation was introduced in \cite{DM1} and gives a nonlinear variant considering the $p$-Laplace equation
\begin{align*}
\int_B \mathcal |\nabla\bfu|^{p-2}\nabla\bfu:\nabla\bfphi\dx=0\quad\forall\bfphi\in C^{\infty}_0(\Omega).
\end{align*}
Here almost solutions in $W^{1,p}(B)$ can be approximated by solutions in the $L^p(B)$-sense. This technique has been improved in \cite{DSV} where an approximation in $W^{1,s}(B)$ for all $s<p$ is possible. Moreover, it applies also to the more general setting of Orlicz spaces. A crucial tool for the approximation result in \cite{DSV} is the Lipschitz truncation method (originally developed in \cite{AF}; it allows to approximate a Sobolev function by a Lipschitz continuous function in a way that they are equal on a large set whose size can be controlled). Based on the $\mathcal A$-harmonic approximation and its generalizations a lot of (partial-) regularity results for nonlinear PDEs have been shown. A nice overview about regularity and irregularity for elliptic problems is given in \cite{Mi}.\\

Let us turn to fluid mechanics. The $p$-Stokes problem - describing the slow stationary flow of a Non-Newtonian fluid \cite{AM},\cite{BAH} - reads as follows: for a given volume force $\bff:\Omega\rightarrow\R^d$ find $(\bfv,\pi)$ such that
\begin{align}
  \left\{\begin{array}{rlc}
      \Div\bfS(\ep(\bfv)) &=\nabla \pi-\bff& \mbox{in $B$,}\\
      \Div \bfv&=0\qquad& \mbox{in $B$,}\\
      \bfv&=\bfv_0\qquad\quad& \mbox{ \,on
        $\partial B$.}\end{array}\right.\label{1.1}
\end{align}
Here the nonlinear tensor $\bfS$ satisfies the $p$-growth condition
$$\lambda(1+|\ep|^2)^{\frac{p-2}{2}}|\bfxi|^2\leq
D\bfS(\ep)(\bfxi,\bfxi)\leq
\Lambda(1+|\ep|^2)^{\frac{p-2}{2}}|\bfxi|^2$$ for $\ep,\bfxi\in
\s^d$ with $\lambda,\Lambda>0$ and $p\in(1,\infty)$. The problem firstly appears in the mathematical literature in the work of Ladyshenskaya and Lions (see \cite{La}-\cite{La3} and \cite{Li}). Regularity results are shown in \cite{Fu4}, \cite{KMS}, \cite{MNRR}, \cite{NW}, \cite{Se} and others.
In contrast to classical problems in nonlinear PDEs like the $p$-Laplace equation we have only control over the symmetric part $\ep(\bfv):=\frac{1}{2}\big(\nabla\bfv+\nabla\bfv^
T\big)$ of the gradient and more important we have the side condition $\Div\bfv=0$. A corresponding approximation theory is developed in \cite{BrDF} and approximates an almost solutions $\bfv\in W^{1,p}_{0,\Div}(B)$ to the $\mathcal A$-Stokes problem by a solution in the $W^{1,s}$-sense for all $s<p$. Let us consider the function spaces 
\begin{align*}
\LD(B)&:=\left\{\bfu\in L^1(B):\,\,\ep(\bfu)\in L^1(B)\right\},\quad
\LD_{0}(B):=\left\{\bfu\in \LD(B):\,\,\bfu|_{\partial B}=0\right\},\\
\LD_{\Div}(B)&:=\left\{\bfu\in \LD(B):\,\,\Div\bfu=0\right\},\quad
\LD_{0,\Div}(B):=\LD_{0}(B)\cap \LD_{\Div}(B),
\end{align*}
where $\bfu|_{\partial B}$ has to be understood in the $\mathcal H^{d-1}(\partial B)$-sense (see \cite{ts}).
A function $\bfv\in \LD_{\Div}(B)$ is a solution to the $\mathcal A $-Stokes problem iff
\begin{align}
\int_B \mathcal A(\ep(\bfv),\ep(\bfphi))\dx=0\quad\forall\bfphi\in C^{\infty}_{0,\Div}(B).
\end{align}
It is an almost solution if we have
\begin{align}
\Big|\dashint_B \mathcal A(\ep(\bfv),\ep(\bfphi))\dx\,\Big|\leq \delta\,\dashint_{B}|\ep(\bfv)|\dx\|\ep(\bfphi)\|_\infty\quad \forall\bfphi\in C^{\infty}_{0,\Div}(B)
\end{align}
with some $\delta\ll1$. Note that the formulation above is the weakest possible as it only requires $\bfv\in \LD(B)$. The approximation result in \cite{BrDF} is based on the solenoidal Lipschitz truncation combined with results developed in \cite{DLSV} for the $\mathcal A$-harmonic approximation in Orlicz spaces.\\
In order to study regularity properties of nonlinear parabolic equations Duzaar and Mingione \cite{DM05} introduce the $\mathcal A$-caloric approximation which compares almost solutions
to the $\mathcal A$-heat equation with its solutions. We call a function $\bfu\in L^1(I;W^{1,1}(B))$ with
\begin{align}\label{eq:neq}
\int_Q \bfu\cdot\partial_t\bfphi\dxt-\int_Q \mathcal A(\nabla\bfu,\nabla\bfphi)\dxt=0\quad\forall\bfphi\in C^{\infty}_0(Q)
\end{align}
$\mathcal A$-caloric on $Q:=I\times B$, where $I\subset\R$ is a bounded interval. If the left hand side of \eqref{eq:neq} is small we talk about an almost $\mathcal A$-caloric function. In \cite{DM05} it is
shown that every almost $\mathcal A$-caloric function $\bfu\in L^p(I;W^{1,p}(B))$ can be approximated by a $\mathcal A$-caloric function in the parabolic $L^p$-sense. This is used to establish partial regularity results for nonlinear parabolic systems (see \cite{DuMiSt} for an overview). In contrast to the elliptic setting there is not so much literature available.\\
The aim of the present paper is to develop an approximation theory for non-stationary problems in fluid mechanics in the fashion of the $\mathcal A$-caloric approximation. Let us be a little bit more precise: A function $\bfv\in L^1(I;\LD_{\Div}(B))$ is a solution to the non-stationary $\mathcal A $-Stokes problem iff
\begin{align}
\int_Q \bfv\cdot\partial_t\bfphi\dxt-\int_Q \mathcal A(\ep(\bfv),\ep(\bfphi))\dx=0\quad\forall\bfphi\in C^{\infty}_{0,\Div}(Q).
\end{align}
It is an almost solution if we have
\begin{align}
\Big|\dashint_Q \bfv\cdot\partial_t\bfphi\dxt-\dashint_Q \mathcal A(\ep(\bfv),\ep(\bfphi))\dx\,\Big|\leq \delta\,\dashint_{Q}|\ep(\bfv)|\dx\|\ep(\bfphi)\|_\infty
\end{align}
for all $\bfphi\in C^{\infty}_{0,\Div}(Q)$ with some $\delta\ll1$. The main results of this paper (see Theorem \ref{Astokes} in section 4) states that every almost solution $\bfv\in L^p(I;W^{1,p}_{0,\Div}(B))$ to the non-stationary $\mathcal A$-Stokes problem can be approximated by a solution in the $L^s(I;W^{1,s}(B))$-sense for all $s<p$. So, we extend the result from \cite{BrDF} to non-stationary flows. Again we are able to work with the weakest formulation of almost solutions.
The main tool is the solenoidal Lipschitz truncation for parabolic PDEs which was recently developed in \cite{BrDS}. We present a version of it which is appropriate for our purposes in section 3. In addition to the results from \cite{BrDS} we show an approximation result for the distributional time derivative of the Lipschitz truncation (see Thm. \ref{cor:appl} d) ). In section 2 we present an $L^q$-theory for the non-stationary $\mathcal A$-Stokes problem in divergence form. It might not be surprising for experts but it is hard to find a reference in literature. \\
The investigation of the parabolic $\mathcal A$-Stokes approximation is a first step towards a partial regularity theory for nonlinear Stokes systems (results for the nonstationary $p$-Navier-Stokes equations are shown in \cite{Se} under certain restrictions on $p$). However the second step is still missing: an appropriate Caccioppoli-type inequality. For the heat equation (and also for parabolic systems with quadratic growth) it can be shown that solutions satisfy
\begin{align}\label{cacc}
\sup_{I_r}\dashint_{B_r}|\bfu|^2\dx+\dashint_{Q_r}|\nabla\bfu|^2\dxt\leq\,c\,\dashint_{Q_{2r}}\Big|\frac{\bfu}{r}\Big|^2\dxt.
\end{align} 
Here $Q_r=I_r\times B_r$ is a parabolic cube with radius $r$. A corresponding version for parabolic systems with $p$-growth also exists (see \cite{DuMiSt}). Due to the appearance of the pressure term it seems to be not possible to show such an estimate for Stokes systems - not even in the linear case. Some modified variants of \eqref{cacc} can be proved, but the known versions are not strong enough to show partial regularity.

\newpage

\section{$L^q$-theory for the $\mathcal A$-Stokes system}
\label{sec:reg}
The aim of this section is to present regularity results for the (non-stationary) $\mathcal A$-Stokes system depending on the right hand side (in divergence form). Let us fix for this section a bounded domain $\Omega\subset \R^d$ with $C^2$-boundary and a time interval $(0,T).$
The $\mathcal A$-Stokes problem (in the pressure-free formulation) with right hand side $\bff\in L^1(\Omega)$ reads as: find $\bfv\in \LD_{0,\Div}(\Omega)$ such that
\begin{align}\label{eq:Af}
\int_\Omega \mathcal A(\ep(\bfv),\ep(\bfphi))\dx=\int_{\Omega} \bff\cdot\bfphi\dx\quad\text{for all }\bfphi\in C^\infty_{0,\Div}(\Omega).
\end{align}
The right hand side can also be given in divergence form, i.e.
\begin{align}\label{eq:AF}
\int_\Omega \mathcal A(\ep(\bfv),\ep(\bfphi))\dx=\int_\Omega \bfF:\nabla\bfphi\dx\quad\text{for all }\bfphi\in C^\infty_{0,\Div}(\Omega)
\end{align}
for $\bfF\in L^1(\Omega)$. For certain purposes it is convenient to discuss the problem with a fixed divergence. To be precise for $g\in L^1_0(\Omega)$, where
$$L^q_0(\Omega):=\Big\{f\in L^q(\Omega):\,\, \int_\Omega f\dx=0\Big\},\quad q\geq1,$$ 
we are seeking for a function $\bfv\in \LD_{0}(\Omega)$ with $\Div\bfv=g$ satisfying \eqref{eq:Af} or \eqref{eq:AF}.
We have the following $L^q$-estimates.
\begin{lems}
\label{lems:Aq}
\begin{enumerate}
Let $\Omega\subset\R^d$ be a bounded $C^2$-domain, $1<q<\infty$ and suppose that \eqref{eq:ell} holds.
\item[a)] Let $\bff\in L^q(\Omega)$ and $g\in W^{1,q}(\Omega)$ with $\int_\Omega g\dx=0$. Then there is a unique solution $\bfw\in W^{2,q}\cap W^{1,q}_{0}(\Omega)$ to \eqref{eq:Af} such that $\Div\bfv=g$ and 
\begin{align*}
\dashint_{\Omega}|\nabla^2\bfw|^q\dx\leq \,c\,\dashint_{\Omega}|\bff|^q\dx+c\,\dashint_{\Omega}|\nabla g|^q\dx,
\end{align*}
where $c$ only depends on $\mathcal A$ and $q$.
\item[b)] Let $\bfF\in L^q(\Omega)$ and $g\in L^q_0(\Omega)$. Then there is a unique solution $\bfw\in W^{1,q}_{0}(\Omega)$ to (\ref{eq:AF}) such that $\Div\bfv=g$ and 
\begin{align*}
\dashint_{\Omega}|\nabla\bfw|^q\dx\leq \,c\,\dashint_{\Omega}|\bfF|^q\dx+c\,\dashint_{\Omega}|g|^q\dx,
\end{align*}
where $c$ only depends on $\mathcal A$ and $q$.
\end{enumerate}
\end{lems}
In case $\mathcal A=I$ both parts follow from \cite{AmGi}, Thm 4.1. However, the main tool in \cite{AmGi} is the theory from
\cite{ADN1,ADN2} where very general linear systems are investigated. Hence it is clear that the results also hold in case of an arbitrary
elliptic tensor $\mathcal A$.

Now we turn to the parabolic problem and the first result is a local $L^q$-estimate for weak solutions. In case of the $\mathcal A$-heat system this follows from the continuity of the corresponding semigroup (see \cite{Sh}). It is also known for the non-stationary Stokes-system (see \cite{So} and \cite{Gi}). In case of the Stokes system one has $\mathcal A=I$ such that $$\Div \Div\mathcal A\bfv=\Div\Delta\bfv=\Delta\Div\bfv=0$$
at least in the sense of distributions. As this does not hold for general $\mathcal A$, the known methods for the Stokes-system do not apply in our setting.
\begin{them}
\label{thm:Lqlocal}
Let $\bff\in L^q(Q_0)$ for some $q>2$, where $Q_0:=(0,T)\times\R^d$ and suppose that \eqref{eq:ell} holds. Let $\bfv\in L^1((0,T);\LD_{\Div}(\R^d))$ be a weak solution to
\begin{align}
\int_{Q_0} \bfv\cdot\partial_t\bfphi\dxt-\int_{Q_0} \mathcal A(\ep(\bfv),\ep(\bfphi))\dx=\int_{Q_0}\bff\cdot\bfphi\dxt
\end{align}
for all $\bfphi\in C^{\infty}_{0,\Div}([0,T)\times \R^d)$. Then we have $\nabla^2\bfv\in L^q_{loc}(Q_0)$ and there holds
\begin{align*}
\int_0^T\int_{B}|\nabla^2\bfv|^q\dxt\leq\,c_B\, \int_{Q_0}|\bff|^q\dxt
\end{align*}
for all balls $B\subset\R^d$.
\end{them}
\begin{proof}
The main ingredient is the proof of the following auxiliary result which has been used in a similar version
 in \cite{ByWa}. A main difference is that we have to take into account the divergence-free constraint.\\
We say that $Q' = I' \times B' \subset \setR \times
\Rdr$ is a parabolic cylinder if $r_{I'} =
r_{B'}^2$. For $\kappa>0$ we define the scaled cylinder $\kappa Q' :=
(\kappa I') \times (\kappa B')$.  By $\mathcal{Q}$ we denote
the set of all parabolic cylinders.  We define the
parabolic maximal operators $\mathcal M$ and $\mathcal M_s$ for $s \in
[1,\infty)$ by
\begin{align*}
  (\mathcal M f)(t,x) &:= \sup_{ Q' \in
    \mathcal{Q}\,:\,(t,x) \in Q'} \dashint_{Q'} \abs{f(\tau,y)} \dtau \dy,
  \\
  \mathcal M_s f(t,x) &:= \big(\mathcal M(\abs{f}^s)(t,x)\big)^{\frac 1s}.
\end{align*}
It is standard\cite{Ste93} that for all $q \in (s,\infty]$
\begin{align}
  \label{eq:Mascont}
  \norm{\Ma_s f}_{L^q(\R^{n+1})} &\leq c\, \norm{f}_{L^q(\R^{n+1})}.
\end{align}
\begin{itemize}
\item[i)] We start with interior estimates. Let $Q_r:=Q_r(x_0,t_0):=(t_0-r^2,t_0+r^2)\times B_r(x_0)$ be a parabolic cylinder such that $4Q_r\subset Q_0$. We claim the following: There is a constant $N_1>0$ such that for every $\varepsilon>0$ there is $\delta=\delta(\varepsilon)>0$ such that\footnote{Since $\bff\in L^2(0,T;L^2_{loc}(\R^d))$ the standard interior regularity theory implies $\nabla^2\bfv\in L^2(0,T;L^2_{loc}(\R^d))$ and  $\partial_t\bfv\in L^2(0,T;L^2_{loc}(\R^d))$.}
\begin{align}\label{imp1}
\begin{aligned}
&\mathcal L^{d+1}\big(Q_r\cap \set{\mathcal M (|\bff|)^2>N_1^2}\big)\geq \varepsilon\mathcal L^{d+1}(Q_r)\\
\Rightarrow&\,\, Q_r\subset \set{\mathcal M(|\nabla^2\bfv|^2)>1}\cup \set{\mathcal M(|\bff|^2)>\delta^2}.
\end{aligned}
\end{align}
Let us assume for simplicity that $r=1$. In fact, we will establish \eqref{imp1} by showing 
\begin{align}\label{imp2}
\begin{aligned}
&Q_1\cap \set{\mathcal M(|\nabla^2\bfv|^2)\leq 1}\cap \set{\mathcal M(|\bff|^2)\leq \delta^2}\neq\emptyset\\
\Rightarrow&\,\,\mathcal L^{d+1}\big(Q_1\cap \set{\mathcal M (|\nabla\bfv|)^2>N_1^2}\big)< \varepsilon\mathcal L^{d+1}(Q_1)
\end{aligned}
\end{align}
and applying a simple scaling argument.
In order to show \eqref{imp2} we compare $\bfv$ with a solution to a homogeneous problem on $Q_4=(t_0^4-4^2,t_0^4+4^2)\times B_4(x_0^4)\subset Q_0$ (with the same boundary data) which is smooth in the interior. So let us define $\bfh$ as the unique solution to
\begin{align}
  \left\{\begin{array}{rlc}
     \partial_t\bfh- \Div\mathcal A(\ep(\bfh)) &=\nabla \pi_{\bfh}& \mbox{in $Q_4$,}\\
      \Div \bfh&=0\qquad& \mbox{in $Q_4$,}\\
      \bfh&=\bfv\qquad\quad& \mbox{ \,on
        $I_4\times\partial B_4$,}\\
\bfh(t_0^4,\cdot)&=\bfv(t_0^{4},\cdot)\qquad\quad& \mbox{ \,in
        $B_4$.}
\end{array}\right.\label{eq:h}
\end{align}
We test the difference of both equations with $\bfv-\bfh$. This yields by the ellipticity of $\mathcal A$
\begin{align*}
\sup_{t\in I_4}\int_{B_4}|\bfv(t)-\bfh(t)|^2\dx&+\int_{Q_4}|\ep(\bfv)-\ep(\bfh)|^2\dxt\\&\leq \,c\,\int_{Q_4}|\bff|^2\dxt+\,c\,\int_{Q_4}|\bfv-\bfh|^2\dxt.
\end{align*}
An application of Korn's inequality and Gronwall's lemma implies
\begin{align}\label{eq:cacc1}
\sup_{t\in I_4}\int_{B_4}|\bfv(t)-\bfh(t)|^2\dx+\int_{Q_4}|\nabla\bfv-\nabla\bfh|^2\dxt\leq \,c\,\int_{Q_4}|\bff|^2\dxt.
\end{align}
First we insert $\partial_t(\bfv-\bfh)$ which yields similarly
\begin{align}\label{eq:cacc1t}
\int_{Q_4}|\partial_t(\bfv-\bfh)|^2\dx+\sup_{t\in I_4}\int_{B_4}|\nabla\bfv-\nabla\bfh|^2\dx\leq \,c\,\int_{Q_4}|\bff|^2\dxt.
\end{align}
We can introduce the pressure terms $\pi_\bfv,\pi_{\bfh}\in L^2(I_4,L^2_0(B_4))$ in the equations for $\bfv$ and $\bfh$ and show
\begin{align}\label{eq:cacc1pi}
\int_{Q_4}|\pi_\bfv-\pi_{\bfh}|^2\dx\leq \,c\,\int_{Q_4}|\bff|^2\dxt.
\end{align}
Estimate \eqref{eq:cacc1+pi} can be shown by using the \Bogovskii-operator
introduced in \cite{Bog}. It is a solution operator to the divergence equation on a bounded Lipschitz domain $\Omega$ with respect to zero boundary conditions. The corresponding \Bogovskii-operator $\Bog_\Omega$ is continuous from
 $L^{2}_0(\Omega)\rightarrow W^{1,2}_0(\Omega).$
Setting $\Bog=\Bog_{B_4}$ we gain
due to \eqref{eq:cacc1+t} for any $\varphi\in C^\infty_0(Q_4)$ that
\begin{align*}
\int_{Q_4}&(\pi_\bfv-\pi_{\bfh})\varphi\dxt=\int_{Q_4}(\pi_\bfv-\pi_{\bfh})\Div \Bog(\varphi-\varphi_{B_4})\dxt\\
&=\int_{Q_4}\mathcal A\big(\ep(\bfv-\bfh),\ep\big( \Bog(\varphi-\varphi_{B_4})\big)\big)\dxt+\int_{Q_4}\bff\cdot \Bog(\varphi-\varphi_{B_4})\big)\dxt\\
&-\int_{Q_4}\partial_t(\bfv-\bfh)\cdot \Bog(\varphi-\varphi_{B_4})\dxt\\
&\leq \,c\,\Big(\|\nabla(\bfv-\bfh) \|_2+\|\bff\|_2+\|\partial_t(\bfv-\bfh)\|_2\Big)\big\|\nabla\Bog(\varphi-(\varphi)_{B_4})\big\|_{2}\\
&\leq \,c\,\bigg(\int_{Q_4}|\bff|^2\dxt\bigg)^{\frac{1}{2}}\bigg(\int_{Q_4}|\varphi|^2\dxt\bigg)^{\frac{1}{2}}.
\end{align*}
Now we choose a cut off function $\eta\in C^\infty_0(B_4)$ with $0\leq\eta\leq1$ and $\eta\equiv1$ on $B_3$.
We insert $\partial_\gamma(\eta^2\partial_\gamma(\bfv-\bfh))$in the equation for $\bfv-\bfh$ and sum over $\gamma\in\set{1,...,d}$. We gain
\begin{align*}
\sup_{t\in I_4}\int_{B_4}&\eta^2|\nabla(\bfv-\bfh)|^2\dx+\int_{Q_4}\eta^2|\nabla\ep(\bfv)-\nabla\ep(\bfh)|^2\dxt\\&\leq \,c\,\int_{Q_4}\bff\cdot\partial_\gamma(\eta^2\partial_\gamma(\bfh-\bfh))\dxt+\,c(\nabla\eta)\,\int_{Q_4}|\nabla\bfv-\nabla\bfh|^2\dxt
\\&+c\,\int_{Q_4}(\pi-\pi_\bfh)\cdot\partial_\gamma(\nabla\eta^2\cdot\partial_\gamma(\bfh-\bfh))\dxt
\end{align*}
We estimate the term involving $\bff$ by
\begin{align*}
\int_{Q_4}&\bff\cdot\partial_\gamma(\eta^2\partial_\gamma(\bfv-\bfh))\dxt\\&\leq\,\,c(\kappa)\int_{Q_4}|\bff|^2\dxt+\kappa\int_{Q_4}\eta^2|\nabla^2\bfv-\nabla^2\bfh|^2\dxt+c(\nabla\eta)\,\int_{Q_4}|\nabla\bfv-\nabla\bfh|^2\dxt,
\end{align*}
where $\kappa>0$ is arbitrary.
The term involving $\pi-\pi_\bfh$ can be estimated in the same fashion.
Choosing $\kappa>0$ small enough and using the inequality $|\nabla^2\bfu|\leq \,c\,|\nabla\ep(\bfu)|$ as well as \eqref{eq:cacc1}--\ref{eq:cacc1pi} shows
\begin{align}\label{eq:cacc2}
\sup_{t\in I_4}\int_{B_3}|\nabla(\bfv-\bfh)(t)|^2\dx+\int_{I_4}\int_{B_3}|\nabla^2\bfv-\nabla^2\bfh|^2\dxt\leq \,c\,\int_{Q_4}|\bff|^2\dxt.
\end{align}
Now, let us assume that (\ref{imp2})$_1$ holds. Then there is a point $(t_0,x_0)\in Q_1$ such that
\begin{align}\label{eq:2.28}
\dashint_{Q_\sigma(t_0,x_0)}|\nabla^2\bfv|^2\dxt\leq 1,\quad \dashint_{Q_\sigma(t_0,x_0)}|\bff|^2\dxt\leq \delta^2\quad \forall \sigma>0.
\end{align}
Since $Q_4\subset Q_6(t_0,x_0)$ we have
\begin{align}\label{eq:3.29}
\int_{Q_4}|\nabla^2\bfv|^2\dxt\leq c,\quad \int_{Q_4}|\bff|^2\dxt\leq \,c\delta^2.
\end{align}
As $\bfh$ is smooth we know that 
\begin{align}\label{eq:3.33}
N_0^2:=\sup_{Q_3}|\nabla^2\bfh|^2<\infty.
\end{align}
From this we aim to conclude that
\begin{align}\label{eq:3.34}
Q_1\cap \set{\mathcal M(|\nabla^2\bfv|^2)>N_1^2}\subset Q_1\cap \set{\mathcal M(\chi_{Q_3}|\nabla^2\bfv-\nabla^2\bfh|^2)>N_0^2}
\end{align}
for $N_1^2:=\max\set{4N_0^2,2^{d+2}}$. To establish \eqref{eq:3.34} suppose that
\begin{align}\label{eq:3.35}
(t,x)\in Q_1\cap \set{\mathcal M(\chi_{Q_3}|\nabla^2\bfv-\nabla^2\bfh|^2)\leq N_0^2}.
\end{align}
If $\sigma\leq 2$ we have $Q_{\sigma}(t,x)\subset Q_3$ and gain by (\ref{eq:3.33})
\begin{align*}
\dashint_{Q_\sigma(t,x)}&|\nabla^2\bfv|^2\dxt\\&\leq 2\,\dashint_{Q_\sigma(t,x)}\chi_{Q_3}|\nabla^2\bfv-\nabla^2\bfh|^2\dxt+2\,\dashint_{Q_\sigma(t,x)}\chi_{Q_3}|\nabla^2\bfh|^2\dxt\\
&\leq\,4N_0^2.
\end{align*}
If $\sigma\geq2$ we have by \eqref{eq:2.28}
\begin{align*}
\dashint_{Q_\sigma(t,x)}&|\nabla^2\bfv|^2\dxt\leq 2^{d+2}\dashint_{Q_{2\sigma}(t_0,x_0)}|\nabla^2\bfv|^2\dxt\leq 2^{d+2}.
\end{align*}
Combining the both cases yields \eqref{eq:3.34}. This implies together with the continuity of the maximal function on $L^2(\R^2)$,  \eqref{eq:cacc2} and \eqref{eq:3.29}
\begin{align*}
\mathcal L^{d+1}\big(Q_1\cap \set{\mathcal M (|\nabla^2\bfv|^2)>N_1^2}\big)&\leq \mathcal L^{d+1}\big(Q_1\cap \set{\mathcal M (\chi_{Q_3}|\nabla^2\bfv-\nabla^2\bfh|^2)>N_0^2}\big)\\
&\leq\,\frac{c}{N_0^2}\int_{Q_3}|\nabla^2\bfv-\nabla^2\bfh|^2\dxt\\
&\leq\,\frac{c}{N_0^2}\int_{Q_3}|\bff|^2\dxt
\leq \frac{c}{N_0^2}\delta^2\\
&= \varepsilon\mathcal L^{d+1}(Q_1),
\end{align*}
choosing $\delta:=c^{-1/2}N_0\sqrt{\varepsilon}$.
So we have shown (\ref{imp2}) which yields (\ref{imp1}) by a scaling argument.\\
If (\ref{imp1})$_1$ holds then we have
\begin{align*}
\mathcal L^{d+1}&\big(Q_r\cap \set{\mathcal M (|\nabla^2\bfv|)^2>N_1^2}\big)\\&\leq \,\varepsilon\mathcal L^{d+1}\Big(Q_r\cap\set{\mathcal M(|\nabla^2\bfv|^2)>1}\cup \set{\mathcal M(|\bff|^2)>\delta^2}\Big)\\
&\leq \,\varepsilon\Big(\mathcal L^{d+1}\big(Q_r\cap\set{\mathcal M(|\nabla^2\bfv|^2)>1}\big)+\mathcal L^{d+1}\big(Q_r\cap \set{\mathcal M(|\bff|^2)>\delta^2}\big)\Big).
\end{align*}
Multiplying the equation for $\bfv$ by some small number $\varrho=\varrho(\|\bff\|_{q},\|\nabla^2\bfv\|_{2})$ we can assume that
\begin{align}
\mathcal L^{d+1}\big(Q_r\cap\set{\mathcal M(|\nabla^2\bfv|^2)>N_1^2}\big)<\varepsilon.
\end{align}
By induction we can establish that
\begin{align*}
\mathcal L^{d+1}&\big(Q_r\cap \set{\mathcal M (|\nabla^2\bfv|^2)>N_1^{2k}}\big)\\
&\leq \,{\varepsilon}^k\mathcal L^{d+1}\big(Q_r\cap\set{\mathcal M(|\nabla^2\bfv|^2)>1}+c\,\sum_{i=1}^k{\varepsilon}^i\mathcal L^{d+1}\big(Q_r \cap\set{\mathcal M(|\bff|^2)>\delta^2 N_1^{2(k-i)}}\big).
\end{align*}
In the induction step one has to introduce $\bfu_1:=\frac{\bfu}{N_1}$ which is a solution to the $\mathcal A$-Sokes problem with right hand side $\bff_1:=\frac{\bff}{N_1}$.
Now we will show $\mathcal M(\nabla^2\bfv)\in L^q(Q_r)$ (which implies $\nabla^2\bfv\in L^q(Q_r)$). This follows if we can prove that
\begin{align}\label{july23b}
\sum_{k=1}^\infty (N_1^2)^{qk}\mathcal L^{d+1}\big(Q_r\cap\set{\mathcal M(|\nabla^2\bfv|^2)>N_1^{2k}}\big)<\infty,
\end{align}
see \cite{C}.
Since $\bff\in L^q(Q_0)$ and $q>2$ we have  $\mathcal M(|\bff|^2)\in L^{q/2}(Q_0)$ and hence
\begin{align*}
&\sum_{k=1}^\infty (N_1^2)^{qk}\mathcal L^{d+1}\big(Q_r\cap\set{\mathcal M(|\nabla^2\bfv|^2)>N_1^{2k}}\big)\\
&\leq \,c\,\sum_{k=1}^\infty (N_1^2)^{qk}\varepsilon^k\mathcal L^{d+1}\big(Q_r\cap\set{\mathcal M(|\nabla^2\bfv|^2)>1}\big)\\
&+\,c\,\sum_{k=1}^\infty (N_1^2)^{qk}\sum_{i=1}^k\varepsilon^i\mathcal L^{d+1}\big( Q_r\cap\set{\mathcal M(|\bff|^2)>\delta^2 N_1^{2(k-i)}}\big)\\
&\leq \,c_r\,\sum_{k=1}^\infty (\varepsilon N_1^{2})^{qk}+\,c\,\sum_{i=1}^\infty \varepsilon^i(N_1^2)^{qi}\sum_{k=i}^\infty(N_1^2)^{q(k-i)}\mathcal L^{d+1}\big( Q_r\cap\set{\mathcal M(|\bff|^2)>\delta^2 N_1^{2(k-i)}}\big)\\
&\leq \,c_r\,\sum_{k=1}^\infty (\varepsilon N_1^q)^{2k}.
\end{align*}
If we choose $\varepsilon N_1^q<1$ the sum in \eqref{july23b} is converging and we have $\nabla^2\bfv\in L^q(Q_r)$. Since the mapping $\bff\mapsto\nabla^2\bfv$ is linear we gain the desired estimate
\begin{align}\label{eq:Lqint}
\int_{Q_r}&|\nabla^2\bfv|^q\dxt\leq\,c_r\,\int_{Q_0}|\bff|^q\dxt.
\end{align}
\item[ii)] Now let $Q_1$ be a cylinder such that $4Q_1\cap (-\infty,0]\times\R^{d}\neq\emptyset$. Moreover, assume that $Q_1\cap Q_0\neq\emptyset$. We consider the solution $\tilde{\bfh}$ to 
\begin{align}
  \left\{\begin{array}{rlc}
     \partial_t\tilde{\bfh}- \Div\mathcal A(\ep(\tilde{\bfh})) &=\nabla \pi_{\tilde{\bfh}}& \mbox{in $\tilde{Q}_4$,}\\
      \Div \tilde{\bfh}&=0\qquad& \mbox{in $\tilde{Q}_4$,}\\
      \tilde{\bfh}&=\bfv\qquad\quad& \mbox{ \,on
        $\tilde{I}_4\times\partial B_4$,}\\
\tilde{\bfh}(t_0^4,\cdot)&=0\qquad\quad& \mbox{ \,in
        $B_4$,}
\end{array}\right.\label{eq:h}
\end{align}
where $\tilde{I}_m:=I_m\cap(0,T)$ and $\tilde{Q}_m:=\tilde{I}_m\times B_m$. We can establish a variant of (\ref{eq:cacc2}) on $\tilde{Q}_4$. Now we have $\sup_{\tilde{Q}_3}|\nabla^2\tilde{\bfh}|^2<\infty$ due to the smooth initial datum of $\tilde{\bfh}$ (recall that $\bfv(0,\cdot)=0$ a.e.). So we can finish the proof as before and gain  $\nabla^2\bfv\in L^q(Q_1)$. This implies again (\ref{eq:Lqint}).
\item[iii)] The situation $4Q_1\cap [T,\infty)\times\R^{d}\neq\emptyset$ is uncritical again and we can assume that ii) and iii) do not occur for the same cylinder (by choosing sufficiently small cubes)
\end{itemize}
Covering the set $(0,T)\times B$ by smaller cylinders and combing i)-iii) yield the desired estimate.
\end{proof}

\begin{corollary}
\label{cor:Lqlocal}
Under the assumptions of Theorem \ref{thm:Lqlocal} we have for all balls $B\subset\R^d$ the following estimates for some constant $c_B$ which does not depend on $T$.
\begin{itemize}
\item[a)] There holds
\begin{align*}
\int_0^T\int_{B}\bigg(\Big|\frac{\bfv}{T}\Big|^q+\Big|\frac{\nabla \bfv}{\sqrt{T}}\Big|^q+|\nabla^2 \bfv|^q\bigg)\dxt\leq \,c_B\,\int_{Q_0}|\bff|^q\dxt.
\end{align*}
\item[b)] We have $\partial_t\bfv\in L^q(0,T;L^q_{loc}(\R^d))$ together with 
\begin{align*}
\int_0^T\int_{B}|\partial_t\bfv|^q\dxt\,\leq \,c_B\,\int_{Q_0}|\bff|^q\dxt.
\end{align*}
\item[c)] There is $\pi\in L^q((0,T),W^{1,q}_{loc}(\R^d))$ such that
\begin{align*}
\int_{Q_0} \bfv\cdot\partial_t\bfphi\dxt&-\int_{Q_0} \mathcal A(\ep(\bfv),\ep(\bfphi))\dx\\&=\int_{Q_0}\pi\,\Div\bfphi\dxt+\int_{Q_0}\bff\cdot\bfphi\dxt
\end{align*}
for all $\bfphi\in C^{\infty}_{0}([0,T)\times\R^{d})$.
\item[d)] There holds
\begin{align*}
\int_0^T\int_{B}\bigg(\Big|\frac{\pi}{\sqrt{T}}\Big|^q+|\nabla\pi|^q\bigg)\dxt\leq \,c_B\,\int_{Q_0}|\bff|^q\dxt.
\end{align*}
\end{itemize}
\end{corollary}
\begin{proof}
The estimate in a) is a simple scaling argument. Having a solution $\bfv$ defined on $(0,T)\times \R^d$ we gain a solution $\widehat{\bfv}$ on $(0,1)\times\R^d$ by setting
\begin{align*}
\widehat{\bfv}(s,x):=\frac{1}{T}\bfv(Ts,\sqrt{T}x).
\end{align*}
Now we apply Theorem \ref{thm:Lqlocal} to $\widehat{\bfv}$. The constant which appears is independent of $T$. Transforming back to $\bfv$ yields the claimed inequality.\\
b) For $\bfvarphi\in C^\infty_0(Q_0)$ with $\bfphi(t,x)=\tau(t)\bfpsi(x)$ where $\bfpsi\in C^\infty_0(G)$ ($G\subset\R^d$ a bounded Lipschitz domain) we have
\begin{align*}
\int_{Q_0} \bfv\cdot\partial_t\bfphi\dxt&=\int_0^T\partial_t\tau\int_{\R^d} \bfv\cdot\big(\bfpsi_{\Div}+\nabla\bfPsi\big)\dxt\\
&=\int_0^T\partial_t\tau\int_{\R^d} \bfv\cdot\bfpsi_{\Div}\dxt\\
&=\int_{Q_0} \bfv\cdot\partial_t\bfphi_{\Div}\dxt
\end{align*}
where $\bfpsi_{\Div}:=\bfpsi-\nabla\Delta^{-1}_G\Div\bfpsi$ and $\bfPsi:=\Delta^{-1}_G\Div\bfpsi$.\footnote{Here $\Delta_G^{-1}$ is the solution operator to the Laplace equation with zero boundary datum on $\partial G$.} Here we took into account $\Psi|_{\partial G}=0$ as well as $\Div\bfv=0$. Using $\nabla^2\bfv\in L^2(0,T;L^2_{loc}(\R^d))$ we proceed by
\begin{align*}
\int_{Q_0} \bfv\cdot\partial_t\bfphi\dxt
&=\int_0^T\int_{G} \big(\bff-\Div\mathcal A \ep(\bfv)\big)\cdot\bfphi_{\Div}\dxt\\
&\leq\,c\,\bigg(\int_0^T\int_G \big(|\nabla^2 \bfv|^q+|\bff|^q\big)\dxt\bigg)^{\frac{1}{q}}\bigg(\int_0^T\int_G |\bfphi_{\Div}|^{q'}\dxt\bigg)^{\frac{1}{q'}}\\
&\leq\,c\,\bigg(\int_{Q_0} |\bff|^q\dxt\bigg)^{\frac{1}{q}}\bigg(\int_0^T\int_G |\bfphi|^{q'}\dxt\bigg)^{\frac{1}{q'}}.
\end{align*}
In the last step we used the estimate from Theorem \ref{thm:Lqlocal} and continuity of $\nabla\Delta_G^{-1}\Div$ on $L^{q'}(G)$. Duality implies $\partial_t\bfv\in L^q(0,T;L^q_{loc}(\R^d))$ and we can introduce the pressure function $\pi\in L^q(0,T;L^q_{loc}(\R^d))$ as claimed in b) by De Rham's Theorem. Using the equation for $\bfv$ and the estimates in a) and b) we gain
\begin{align*}
\int_{Q}|\nabla\pi|^q\dxt\,\leq \,c\,\int_{Q_0}|\bff|^q\dxt.
\end{align*}
The estimate for $\pi$ in d) follows again by scaling.
\end{proof}

\begin{corollary}
\label{cor:Lqlocal+}
Let $\bff\in L^q(Q_0^+)$ for some $q>2$ where $Q_0^+:=(0,T)\times\R^d_+$, $\R^d_+=\R^d\cap[x_d>0]$ and suppose that \eqref{eq:ell} holds. Let $\bfv\in L^1((0,T);\LD_{\Div}(\R^d_+))$ with $\bfv|_{x_d=0}=0$ be a weak solution to
\begin{align}
\int_{Q_0^+} \bfv\cdot\partial_t\bfphi\dxt-\int_{Q_0^+} \mathcal A(\ep(\bfv),\ep(\bfphi))\dx=\int_{Q_0^+}\bff\cdot\bfphi\dxt
\end{align}
for all $\bfphi\in C^{\infty}_{0,\Div}([0,T)\times \R^d_+)$.
Then the results from Theorem \ref{thm:Lqlocal} and Corollary \ref{cor:Lqlocal} hold for $\bfv$ for all half balls $B^+(z)\subset\R^d$ with $z_d=0$.\footnote{$B^+(z):=B(z)\cap[x_d>0]$}
\end{corollary}
\begin{proof}
We will show a variant of the $L^q$-estimate from Theorem \ref{thm:Lqlocal} on half balls $B^+$, i.e.
\begin{align}\label{eq:B+}
\int_0^T\int_{B^+}|\nabla^2\bfv|^q\dxt\leq \,c_B\,\int_0^T\int_{Q^+}|\bff|^q\dxt.
\end{align}
From this we can follow estimates in the fashion of Corollary \ref{cor:Lqlocal} as done there. In order to establish (\ref{eq:B+}) we will proceed as in the proof of
 Theorem \ref{thm:Lqlocal} replacing all balls with half ball. So let $Q_1\subset\R^{d+1}$ such that $4Q_1\subset Q_0$ (the other situation can be shown along the modifications indicated at the end of the proof of Theorem \ref{thm:Lqlocal}). Moreover, assume that $Q_1=I_1\times B_1(z)$ where $z_d=0$. We compare $\bfv$ with the unique solution $\bfh^+$ to
\begin{align}
  \left\{\begin{array}{rlc}
     \partial_t\bfh^+- \Div\mathcal A(\ep(\bfh^+)) &=\nabla \pi_{\bfh^+}& \mbox{in $Q_4^+$,}\\
      \Div \bfh^+&=0\qquad& \mbox{in $Q_4^+$,}\\
      \bfh^+&=\bfv\qquad\quad& \mbox{ \,on
        $I_4\times\partial B_4^+$,}\\
\bfh^+(0,\cdot)&=\bfv(0,\cdot)\qquad\quad& \mbox{ \,in
        $B_4^+$.}
\end{array}\right.\label{eq:h+}
\end{align}
We gain a version of the estimate \eqref{eq:cacc1} and \eqref{eq:cacc1t} on half-balls. In fact, there holds
\begin{align}\label{eq:cacc1+}
\sup_{t\in I_4}\int_{B_4^+}|\bfv-\bfh^+|^2\dx+\int_{Q_4^+}|\nabla\bfv-\nabla\bfh^+|^2\dxt\leq \,c\,\int_{Q_4^+}|\bff|^2\dxt,\\
\label{eq:cacc1+t}
\int_{Q_4^+}|\partial_t(\bfv-\bfh^+)|^2\dx+\sup_{t\in I_4}\int_{B_4^+}|\nabla\bfv-\nabla\bfh^+|^2\dx\leq \,c\,\int_{Q_4^+}|\bff|^2\dxt.
\end{align}
We can introduce the pressure terms $\pi_\bfv,\pi_{\bfh^+}\in L^2(I_4,L^2_0(B_4^+))$ in the equations for $\bfv$ and $\bfh^+$ and show
\begin{align}\label{eq:cacc1+pi}
\int_{Q_4^+}|\pi_\bfv-\pi_{\bfh^+}|^2\dx\leq \,c\,\int_{Q_4^+}|\bff|^2\dxt.
\end{align}
This can be done as in the proof of \eqref{eq:cacc1pi} using the \Bogovskii-operator on $B_4^+$.
Estimate \eqref{eq:cacc1+pi} can be shown by using the \Bogovskii-operator
introduced in \cite{Bog}.
Now we insert $\partial_\gamma(\eta^2\partial_\gamma(\bfv-\bfh))$ for $\gamma\in\set{1,...,d-1}$ in the equation for $\bfv-\bfh$. Here we choose $\eta\in C^\infty_0(B_4)$ with $0\leq\eta\leq1$ and $\eta\equiv1$ on $B_3$. This yields together with \eqref{eq:cacc1+}-\eqref{eq:cacc1+pi}
\begin{align}\label{eq:cacc2+}
\int_{Q_3^+}|\tilde{\nabla}\nabla(\bfv-\bfh)|^2\dx\leq \,c\,\int_{Q_4^+}|\bff|^2\dxt,
\end{align}
where $\tilde{\nabla}:=(\partial_1,...,\partial_{d-1})$.
Finally, the only term which is missing is $\partial_d^2(\bfv-\bfh)$. On account of $\Div(\bfv-\bfh)=0$ we gain (cf. \cite{BKR})
\begin{align}\label{eq:dd}
|\partial_d^2(\bfv-\bfh)|&\leq \,c\,\big(|\tilde{\nabla}(\pi_\bfv-\pi_{\bfh^+})|+|\tilde{\nabla}\nabla(\bfv-\bfh)|+|\bff|\big).
\end{align}
So we have to estimate derivatives of the pressure. In fact we have
\begin{align}\label{eq:cacc1+dpi}
\int_{Q_3^+}|\tilde{\nabla}(\pi_\bfv-\pi_{\bfh^+})|^2\dx\leq \,c\,\int_{Q_4^+}|\bff|^2\dxt.
\end{align}
We can show this similarly to the proof of \eqref{eq:cacc1+pi} replacing $\varphi$ by $\partial_\gamma\varphi$ and using \eqref{eq:cacc2+}. 
Combining (\ref{eq:cacc2+})- (\ref{eq:cacc1+dpi}) implies
\begin{align*}
\int_{Q_3^+}|\nabla^2\bfv-\nabla^2\bfh|^2\dx\leq \,c\,\int_{Q_4^+}|\bff|^2\dxt.
\end{align*}
 Moreover, we know $\sup_{Q^+_3}|\nabla^2\bfh^+|^2<\infty$. Note that $\bfh^+=0$ on $Q_3\cap [x_d=0]$. This allows to show $\nabla^2\bfv\in L^q(Q^+_1)$ as in the proof of Theorem \ref{thm:Lqlocal}.  
\end{proof}

\begin{them}
\label{thm:Lqglobal}
Let  $Q:=(0,T)\times\Omega$ with a bounded domain $\Omega\subset\R^d$ having a $C^2$-boundary.
Let $\bff\in L^q(Q)$ for some $q>2$ and suppose that \eqref{eq:ell} holds. Then there is a unique weak solution $\bfv\in L^\infty(0,T;L^2(B))\cap L^q(0,T;W^{1,q}_{0,\Div}(\Omega))$ to
\begin{align}
\int_{Q} \bfv\cdot\partial_t\bfphi\dxt-\int_{Q} \mathcal A(\ep(\bfv),\ep(\bfphi))\dx=\int_{Q}\bff\cdot\bfphi\dxt
\end{align}
for all $\bfphi\in C^{\infty}_{0,\Div}([0,T) \times \Omega)$ such that $\nabla^2\bfv\in L^q(Q)$. Moreover, we have
\begin{align*}
\int_{Q}|\nabla^2\bfv|^q\dxt\leq \,c\,\int_{Q}|\bff|^q\dxt.
\end{align*}
\end{them}
\begin{proof}
Due to the local $L^q$-theory for the whole space problem and the half-space problem which follow from Corollary \ref{cor:Lqlocal} and  Corollary \ref{cor:Lqlocal+}  (with the right scaling in $T$) the proof follows exactly as in \cite{So}, Thm. 4.1, in the case $\mathcal A=I$. Note that $L^q$-estimates for the stationary problem on bounded domains with given divergence are stated in Lemma \ref{lems:Aq}. 
\end{proof}

In order to treat problems with right hand side in divergence form
we consider the $\mathcal A$-Stokes operator 
$$\mathscr A_q:=-\mathcal P_q\Div\mathcal A\big(\ep(\cdot)\big),$$
where $\mathcal P_q$ is the Helmholtz projection from $L^q(B)$ into $L^q_{\Div}(B)$. The latter one is defined by 
\begin{align*}
L^q_{\Div}(B):=\overline{C^\infty_{0,\Div}(B)}^{\|\cdot\|_q}.
\end{align*}
The Helmholtz-projection $\mathcal P_q \bfu$ of a function $\bfu\in L^q(B)$ can be defined as $\mathcal P_q \bfu:=\bfu-\nabla h$, where $h$ is the solution to the Neumann-problem
\begin{align*}
\begin{cases}\Delta h=\Div \bfu\quad\text{on}\quad B,\\
\mathcal N_B\cdot(\nabla h-\bfu)=0\quad\text{on}\quad\partial B.
\end{cases}
\end{align*}
The $\mathcal A$-Stokes operator $\mathscr A_q$ enjoys the same properties than the Stokes operator $\mathbb A_q$ (see for instance \cite{GaSiSo}). 

For the $\mathcal A$-Stokes operator it holds $\mathcal D (\mathscr A_q)=W^{1,q}_{0,\Div}\cap W^{2,q}(\Omega)$ and
\begin{align}\label{eq:socont1}
&\|\bfu\|_{2,q}\leq c_1 \|\mathscr A_q\bfu\|_{q}\leq c_2\|\bfu\|_{2,q},\quad \bfu\in\mathcal D(\mathscr A_q),\\
&\int_\Omega \mathscr A_q \bfu\cdot\bfw\dx=\int_\Omega \bfu\cdot \mathscr A_{q'} \bfw\dx\quad\bfu\in D(\mathscr A_q),\,\,\bfw\in D(\mathscr A_{q'}).\label{eq:socont2}
\end{align}
Inequality \eqref{eq:socont1} is a consequence of Lemma \ref{lems:Aq} a) and the continuity of $\mathcal P_q$.\\
Since $\mathscr A_q$ is positive its root $\mathscr A_q^{\frac{1}{2}}$ is well-defined with $\mathcal D (\mathscr A_q^{\frac{1}{2}})=W^{1,q}_{0,\Div}(\Omega)$ and
\begin{align}\label{eq:sorootcont1}
&\|\bfu\|_{1,q}\leq c_1\|\mathscr A^{\frac{1}{2}}_q\bfu\|_{q}\leq c_2\|\bfu\|_{1,q},\quad \bfu\in\mathcal D(\mathscr A^{\frac{1}{2}}_q),\\
&\int_\Omega \mathscr A^{\frac{1}{2}}_q \bfu\cdot\bfw\dx=\int_\Omega \bfu\cdot \mathscr A^{\frac{1}{2}}_{q'} \bfw\dx\quad\bfu\in D(\mathscr A^{\frac{1}{2}}_q),\,\,\bfw\in D(\mathscr A_{q'}^{\frac{1}{2}}).\label{eq:sorootcont2}
\end{align}
Finally, the inverse operator $\mathscr A_q^{-\frac{1}{2}}:L^q_{\Div}(\Omega)\rightarrow W^{1,q}_{0,\Div}(\Omega)$ is defined and it holds
\begin{align}\label{eq:ainv1}
&\|\nabla\mathscr A^{-\frac{1}{2}}_q\bfu\|_{q}\leq c\|\bfu\|_{q},\quad \bfu\in\mathcal D(\mathscr A^{-\frac{1}{2}}_q),\\
&\int_\Omega \mathscr A^{-\frac{1}{2}}_q \bfu\cdot\bfw\dx=\int_\Omega \bfu\cdot \mathscr A^{-\frac{1}{2}}_{q'} \bfw\dx\quad\bfu\in D(\mathscr A^{-\frac{1}{2}}_q),\,\,\bfw\in D(\mathscr A_{q'}^{-\frac{1}{2}}).\label{eq:ainv2}
\end{align}
From the definition of the square root of an positive self-adjoint operator follows also that
\begin{align*}
\mathscr A_{q'}^{\frac{1}{2}}&:W^{2,q'}\cap W^{1,q'}_{0,\Div}(\Omega)\rightarrow W^{1,q'}_{0,\Div}(\Omega),\\
\mathscr A_{q'}^{-\frac{1}{2}}&:W^{1,q'}_{0,\Div}(\Omega)\rightarrow W^{2,q'}\cap W^{1,q'}_{0,\Div}(\Omega),
\end{align*}
 together with
\begin{align}\label{eq:a1/2}
&\|\nabla\mathscr A^{\frac{1}{2}}_q\bfu\|_{q}\leq c\|\bfu\|_{2,q},\quad \bfu\in W^{2,q'}\cap W^{1,q'}_{0,\Div}(\Omega),\\
&\|\nabla\mathscr A^{-\frac{1}{2}}_q\bfu\|_{q}\leq c\|\bfu\|_{q},\quad \bfu\in W^{1,q'}_{0,\Div}(\Omega).\label{eq:a-1/2}
\end{align}

Finally we state the main result of this section.
\begin{them}
  \label{lem:reg2}
Let  $Q:=(0,T)\times\Omega$ with a bounded domain $\Omega\subset\R^d$ having a $C^2$-boundary and suppose that \eqref{eq:ell} holds.
Let $\bfF\in L^q(Q)$, where $q\in(1,\infty)$. There is a unique solution $\bfw \in L^q(0,T;W^{1,q}_{0,\Div}(\Omega))$ to
  \begin{align}
    \label{eq:vH''}
    \begin{aligned}
   \int_Q \bfw\cdot \partial_t\bfphi\dxt  -&\int_Q \mathcal{A}(\ep(\bfw),\ep(\bfphi)) \dxt
   = \int_Q \bfF:\nabla\bfphi
    \end{aligned}
  \end{align}
  for all $\bfphi \in C^\infty_{0,\Div}([0,T)\times \Omega)$. Moreover we have
\begin{align*}
\dashint_{Q}|\nabla\bfw|^q\dxt\leq \,c\,\dashint_{Q}|\bfF|^q\dxt,
  \end{align*}
  where $c$ only depends on~$\mathcal{A}$ and $q$.
\end{them}
\begin{proof}
Let us first assume that $q>2$. Then Theorem \ref{thm:Lqglobal} applies.
We set $\bff:=\mathscr A_q^{-\frac{1}{2}}\Div\bfF$ which is defined via the duality
\begin{align*}
\int_\Omega \mathscr A^{-\frac{1}{2}}_q\Div\bfF \cdot\bfphi\dx=\int_\Omega \bfF :\nabla \mathscr A^{-\frac{1}{2}}_{q'}\bfphi\dx,\quad \bfphi\in C^\infty_{0,\Div}(\Omega),
\end{align*}
using \eqref{eq:ainv2}. So we gain $\bff\in L^q_{\Div}(\Omega)$ with 
\begin{align}\label{eq:fF}
\|\bff\|_q\leq \,c\,\|\bfF\|_q.
\end{align}
We define $\tilde{\bfw}\in  L^q(0,T;W^{1,q}_{0,\Div}(\Omega))$ as the unique solution to
  \begin{align}
    \label{eq:vH'}
    \begin{aligned}
   \int_Q \tilde{\bfw}\cdot \partial_t\bfphi\dxt  -&\int_Q \mathcal A(\ep(\tilde{\bfw}),\ep(\bfphi)) \dxt=
   \int_Q \bff\cdot\bfphi\dxt
    \end{aligned}
  \end{align}
  for all $\bfphi \in C^\infty_{0,\Div}([0,T)\times \Omega)$. Theorem \ref{thm:Lqglobal} yields $\tilde{\bfw} \in L^q(0,T;W^{2,q}(\Omega))$ and
  \begin{align}\label{eq:wf}
\|\tilde{\bfw}\|_{2,q}\leq \,c\,\|\bff\|_q.
  \end{align}
  We want to return to the original problem and set $\bfw:=\mathscr A^{\frac{1}{2}}_q\tilde{\bfw}$ thus $\bfw\in L^q(0,T;W^{1,q}_{0,\Div}(\Omega))$. 
  Since $\mathscr A_{q'}^{\frac{1}{2}}:W^{1,q'}_{0,\Div}\cap W^{2,q'}(\Omega)\rightarrow W^{1,q'}_{0,\Div}(\Omega)$ we can replace
  $\bfphi$ by $\mathscr A_{q'}^{\frac{1}{2}}\bfphi$ in \eqref{eq:vH'}. This implies using \eqref{eq:sorootcont2} and the definition of $\bff$
   \begin{align*}
   \int_Q \bfw\cdot \partial_t\bfphi\dxt  +\int_Q &\Div\mathcal A(\ep(\tilde{\bfw})):\mathscr A_{q'}^{\frac{1}{2}}\bfphi \dxt
   = \int_Q \bfF:\nabla\bfphi
  \end{align*}
   for all $\bfphi \in C^\infty_{0,\Div}(Q)$. On account of $\mathscr A_{q'}^{\frac{1}{2}}\bfphi\in W^{1,q'}_{0,\Div}(\Omega)$ and $\mathscr A_{q}^{\frac{1}{2}}\tilde{\bfw}\in W^{1,q}_{0,\Div}(\Omega)$ we gain due to \eqref{eq:sorootcont2}
   \begin{align*}
\int_Q \Div\mathcal A(\ep(\tilde{\bfw})):\mathscr A_{q'}^{\frac{1}{2}}\bfphi \dxt   &=\int_Q \mathscr A_{q}\tilde{\bfw}:\mathscr A_{q'}^{\frac{1}{2}}\bfphi \dxt 
 =\int_Q \mathscr A_{q}^{\frac{1}{2}}\tilde{\bfw}:\mathscr A_{q'}\bfphi \dxt\\
 &=\int_Q \bfw\cdot \Div\mathcal A(\ep(\bfphi)) \dxt=-\int_Q \mathcal A(\ep(\bfw),\ep(\bfphi)) \dxt
   \end{align*}
   using \eqref{eq:sorootcont2} and $\bfw\in W^{1,q}_{0,\Div}(\Omega)$. This shows that $\bfw$ is the unique solution to \eqref{eq:vH''}. Moreover,
   we obtain the desired regularity estimate via
   \begin{align*}
\int_{Q}|\nabla\bfw|^q\dxt&\leq \,c\,\int_{Q}|\mathscr A_q^{\frac{1}{2}}\bfw|^q\dxt=\,c\,\int_{Q}|\mathscr A_q\tilde{\bfw}|^q\dxt\\
&\leq\,c\,\int_{Q}|\nabla^2\tilde{\bfw}|^q
\leq \,\,c\,\int_{Q}|\bff|^q\dxt\\&\leq  \,c\,\int_{Q}|\bfF|^q\dxt
  \end{align*}
   as a consequence of \eqref{eq:sorootcont1}, the definition of $\bfw$, (eqref{eq:socont1}, \eqref{eq:wf}, and \eqref{eq:fF}. A simple scaling argument shows that
   the inequality is independent of the diameter of $I$ and $B$. So we have shown the claim for $q>2$.\\
   The case $q=2$ follows easily from a priori estimates and Korn's inequality. So let us
assume that $q<2$.  Duality arguments show that
  \begin{align*}
    \frac{1}{q}\dashint_Q \abs{\nabla\bfw}^q\dxt = \, \sup_{\bfG \in L^{q'}(Q)}
    \bigg[ \dashint_Q \nabla\bfw:\bfG\dxt - \frac{1}{q'}\dashint_Q \abs{
      \bfG}^{q'}\dxt \bigg].
  \end{align*}
  For a given $\bfG\in L^{q'}(Q)$ let $\bfz_\bfG$ be the unique
  $L^{q'}(0,T;W^{1,q'}_{0,\Div}(\Omega))$-solution to
  \begin{align*}
  -\dashint_{Q} \bfz\cdot \partial_t\bfxi\dxt  + \int_Q \mathcal{A}(\ep(\bfz),\ep(\bfxi)) \dxt = \int_Q \bfG
    : \nabla \bfxi \dxt
  \end{align*}
  for all $\bfxi \in C^\infty_{0,\Div}([0,T)\times \Omega)$. Its existence together with the estimate
  \begin{align*}
    \dashint_Q|\nabla \bfz_{\bfG}|^{q'}\dxt\leq
    c\dashint_Q|\bfG|^{q'}\dxt
  \end{align*}
  follows from the first part of the proof as $q'>2$.
This, the density of
  $C^\infty_{0,\Div}(Q)$ and $\partial_t\bfz\in L^{q'}(0,T;W^{-1,q'}_{\Div}(\Omega))$ yield
  \begin{align*}
    \dashint_Q &\abs{\nabla\bfw}^q\dxt\\& \leq \, c\sup_{\bfG \in
      L^{q'}(Q)} \bigg[ \dashint_Q \mathcal{A}(\ep(\bfw),
    \ep(\bfz_\bfG))\dxt+\dashint_{Q} \partial_t\bfz_\bfG\cdot \bfu\dxt - \dashint_Q \abs{ \nabla\bfz_\bfG}^{q'}\dxt
    \bigg]
    \\
    &\leq c\, \sup_{\bfxi \in C^\infty_{0,\Div}(Q)} \bigg[ \dashint_Q
    \mathcal{A}(\ep(\bfw), \ep(\bfxi))\dxt-\dashint_{Q} \bfu\cdot \partial_t\bfxi\dxt- \dashint_Q \abs{\nabla
      \bfxi}^{q'}\dxt\bigg].
  \end{align*}
  The equation for $\bfw$ and Young's inequality imply
    \begin{align*}
    \dashint_Q \abs{\nabla\bfw}^q\dxt
    &\leq \,c\, \sup_{\bfxi \in C^\infty_{0,\Div}(Q)} \bigg[ \dashint_Q
    \bfF:\nabla\bfxi\dxt- \dashint_Q \abs{\nabla
      \bfxi}^{q'}\dxt\bigg]\\
      &\leq\,c\,\dashint_Q|\bfF|^{q}\dxt
  \end{align*}
  and hence the claim.
\end{proof}

\newpage

\section{Solenoidal Lipschitz truncation}
\label{sec:soltrunc}
The purpose of the Lipschitz truncation technique is to approximate a
Sobolev function~$u \in W^{1,p}$ by $\lambda$-Lipschitz
functions~$u_\lambda$ that coincides with~$u$ up to a set of small
measure. The functions $u_\lambda$ are constructed nonlinearly by
modifying~$u$ on the level set of the Hardy-Littlewood maximal
function of the gradient~$\nabla u$.  This idea goes back to Acerbi
and Fusco~\cite{AF}.  Lipschitz truncations are used in various
areas of analysis: calculus of variations, in the existence theory of
partial differential equations, and in the regularity theory. We refer
to~\cite{DMS} for a longer list of references.
The Lipschitz truncation in the context of parabolic PDEs can be found in \cite{KinLew02} and \cite{DRW}. The solenoidal Lipschitz truncation in the non-stationary setting is introduced in \cite{BrDS}. We present a version of it which is appropriate for our purposes. For notational simplicity we assume $d=3$ and refer to \cite{BrDS} (remark 2.1) for the higher dimensional case. In the following let us denote by $Q_0=I_0\times B_0\subset\R\times\R^3$ a parabolic cylinder with radius $r_0$.
\begin{them}
  \label{cor:appl}
Let $\bfu\in L^\sigma(I_0;W^{1,\sigma}_{\div}(B_0))\cap L^\infty(I_0;L^\sigma(B_0))$  for some $\sigma>0$ with $\partial_t\bfu=\Div\bfH$ in $\mathcal D'_{\Div}(Q_0)$ with $\bfH\in L^\sigma(Q_0)$. Then for every $m_0\gg1$ and $\gamma>0$
there exist $\lambda\in[2^{m_0}\gamma,2^{2m_0}\gamma]$ and a function $\bfu_\lambda$ with the following properties
 \begin{enumerate}[label=(\alph{*})]
  \item \label{itm:1} It holds $\bfu_\lambda\in L^\infty(I_0, W^{1,\infty}_{0,\Div}(B_0))$ with $\|\nabla\bfu_\lambda\|_{\infty}\leq c\lambda$.
  \item \label{itm:2} We have
  \begin{align*}
   \lambda^{\sigma} \frac{\mathcal{L}^{d+1}\left(\frac 12 Q_0\cap\set{\bfu_\lambda \not=
          \bfu}\right)}{\abs{Q_0}} &\leq \frac{c}{m_0} \bigg(\dashint_{Q_0}
    r_0^{-\sigma}\abs{\bfu}^{\sigma}+\abs{\nabla \bfu}^{\sigma}\dxt+\dashint_{Q_0} \abs{ \bfH}^{\sigma}\dxt\bigg).
    \end{align*}
  \item \label{itm:3} It holds
\begin{align*}  
  \dashint_{Q_0}
    \abs{\bfu_\lambda}^{\sigma}\dxt&\leq c\bigg(\dashint_{Q_0}
    \abs{\bfu}^{\sigma}+\dashint_{Q_0} r_0^{\sigma}\abs{ \bfH}^{\sigma}\dxt\bigg),\\
  \dashint_{Q_0}\abs{\nabla \bfu_\lambda}^{\sigma}\dxt&\leq c\bigg(\dashint_{Q_0}
    r_0^{-\sigma}\abs{\bfu}^{\sigma}+\abs{\nabla \bfu}^{\sigma}\dxt+\dashint_{Q_0} \abs{ \bfH}^{\sigma}\dxt\bigg).
  \end{align*}
   \item \label{itm:4} We have $\partial_t(\bfu-\bfu_\lambda)\in L^{\sigma'}(\frac 12 I_0,W^{-1,\sigma'}(\frac 12 B_0))$ and
   \begin{align*}
   -\dashint_{\frac 12 Q_0} &(\bfu-\bfu_\lambda)\cdot\partial_t\bfphi\dxt\\
  & \leq c(\kappa)\int_{\frac 12 Q_0} \chi_{\mathcal O_\lambda}|\nabla\bfphi|^{\sigma'}\dxt+\kappa \bigg(\dashint_{Q_0}
    r_0^{-\sigma}\abs{\bfu}^{\sigma}+\abs{\nabla \bfu}^{\sigma}+ \abs{ \bfH}^{\sigma}\dxt\bigg)
   \end{align*}
   for all $\bfphi\in C^\infty_0(\frac 12 Q_0)$ and all $\kappa>0$.
  \end{enumerate}
\end{them}

\begin{proof}

Let $\gamma \in C_0^\infty(\Bz)$ with$ \chi_{\frac 12 \Bz} \leq \gamma
\leq \chi_{\Bz}$. Let $A$ denote the annulus $\Bz \setminus \frac 12 \Bz$.
Then according to \cite{Bog} there exists a \Bogovskii{}
operator $\Bog_A$ from\footnote{$C^\infty_{0,0}$ is the subspace of
  $C^\infty_0$ whose elements have mean value zero.}
$C^\infty_{0,0}(A) \to C^\infty_0(A)$ which is bounded from $L^q_0(A)
\to W^{1,q}_0(A)$ for all $q \in (1,\infty)$ such that $\divergence
\Bog_A = \identity$.  Define
\begin{align*}
  \tilde{\bfu} &:= \gamma \bfu - \Bog_A( \divergence(\gamma \bfu)) =
  \gamma \bfu - \Bog_A( \nabla\gamma \cdot \bfu).
\end{align*}
Then $\divergence \tilde{\bfu}=0$ on $\Iz \times \Bz$ and
$\tilde{\bfu}(t) \in W^{1,\sigma}_0(\Bz)$, so we can extend
$\tilde{\bfu}$ by zero in space to $\tilde{\bfu} \in L^\sigma(I,
W^{1,\sigma}_{\divergence}(\Rdr))$.  Since $\tilde{\bfu}=\bfu$ on $I
\times \frac 12 \Bz$, we have
\begin{align}
  \label{eq:uH2}
  \int_{\Qz} \partial_t \tilde{\bfu} \cdot \bfxi \dxt &= \int_{\Qz}
  \bfH : \nabla \bfxi \dxt \qquad \text{for all } \bfxi \in
  C^\infty_{0,\divergence}(\tfrac 12\Qz).
\end{align}
On the space $W^{1,\sigma}_{\divergence}(\setR^3)$ with $\sigma>1$ we
define the inverse curl operator $\curl^{-1}$ by
\begin{align*}
  \curl^{-1} \bfg := \curl(\Delta^{-1} \bfg) := \curl\bigg(\int_\Rdr
  \frac{-1}{4\pi \abs{x-y}} \bfg(y)\dy\bigg).
\end{align*}
Now, we define pointwise in time
\begin{align*}
  \bfw := \curl^{-1} (\tilde{\bfu}) = \curl^{-1}\big(   \gamma \bfu -
  \Bog_{\Bz \setminus \frac 12 \Bz}( \nabla\gamma \cdot \bfu) \big).
\end{align*}
It follows from Lemma 2.1  in \cite{BrDS}) and continuity properties of $\Bog$ that
\begin{align}
\label{eq:w}
\begin{aligned}
\int_{\frac{1}{2}Q_0}|\nabla\bfw|^{\sigma}\dxt&\leq\,c\,\dashint_{Q_0}
    \abs{\bfu}^{\sigma}\dxt,\\
\int_{\frac{1}{2}Q_0}|\nabla^2\bfw|^\sigma\dxt&\leq\,c\,\bigg(\,\,\dashint_{Q_0}r_0^{-\sigma}
    \abs{\bfu}^{\sigma}\dxt+\int_{Q_0}\abs{\nabla \bfu}^{\sigma}\dxt\bigg).
\end{aligned}
\end{align}
For a ball $B' \subset \Rdr$ and a function $f \in L^s(B')$ we define $\Delta^{-2}_{B'} \Delta f$ as the weak solution $F \in
W^{2,s}_0(B')$ to
\begin{align}
  \label{eq:defbi2}
  \int_{B'}\Delta F\Delta\phi\dx &= \int_{B'}f\Delta\phi\dx \qquad
  \text{for all}\quad\phi \in C^\infty_0(B').
\end{align}
Then $f - \Delta (\Delta^{-2}_{B'} \Delta f)$ is harmonic in~$B'$. We define $\bfz(t) := \bfz_{\frac 12 \Qz}(t) = \Delta \Delta_{\frac 12
  \Bz}^{-2} \Delta \bfw(t)$ for $t \in \frac12 \Iz$, then
\begin{align}
  \label{eq:zt}
  \int_{\Qz} \bfz\cdot \partial _t \Delta\bfpsi \dxt &= \int_{\Qz}
  \bfw\cdot \partial _t \Delta\bfpsi \dxt = - \int_{\Qz} {\bfH}:
  \nabla^2\bfpsi \dxt,
\end{align}
for all $\bfpsi \in C^\infty_0(\tfrac 12 \Qz)$.  Since $\Delta_{\frac
  12 \Bz}^{-2} \bfw(t) \in W^{2,s}_0(\frac 12 \Bz)$, we can extend it
by zero to a function from $W^{2,s}(\Rdr)$. In this sense it is natural
to extend $\bfz(t)$ by zero to a function from $L^s(\Rdr)$. As a consequence of \cite[Lemma 2.3]{BrDS} and (\ref{eq:w}) we have
\begin{align}
\label{eq:z0}
\begin{aligned}
\int_{\frac{1}{2}Q_0}|\nabla\bfz|^{\sigma}\dxt&\leq\,c\,\dashint_{Q_0}
    \abs{\bfu}^{\sigma}\dxt,\\
\int_{\frac{1}{2}Q_0}|\nabla^2\bfz|^{\sigma}\dxt&\leq\,c\,\bigg(\,\,\dashint_{Q_0}r_0^{-{\sigma}}
    \abs{\bfu}^{\sigma}\dxt+\int_{Q_0}\abs{\nabla \bfu}^{\sigma}\dxt\bigg),\\
\int_{\frac{1}{2}Q_0}|\partial_t\bfz|^{\sigma}\dxt&\leq\,c\,\int_{Q_0}\abs{\bfH}^{\sigma}\dxt.
\end{aligned}
\end{align}
For $\lambda,\alpha>0$ and $\sigma>1$ we define
\begin{align}
  \label{eq:defOal}
  \Oal &:= \set{\Mas(\chi_{\frac 13 \Qz} \abs{\nabla^2
      \bfz})>\lambda} \cup \set{\alpha \Mas(\chi_{\frac 13 \Qz}
    \abs{\partial_t\bfz}) > \lambda}.
\end{align}
We decompose $\Oal$ into a family of parabolic Whitney cubes $(Q_i)_{i\in\N}$ an consider a decomposition of unity $(\varphi)_{i\in\N}$ with respect to it as done in \cite{BrDS} after (2.11).
We define
\begin{align*}
  \mathcal{I} := \set{ i \,:\, Q_i \cap \tfrac 14 \Qz \neq \emptyset}.
\end{align*}
Taking $\lambda$ large enough, the continuity of the maximal function (\ref{eq:Mascont}) implies $Q_i \subset \frac 13 \Qz$ (and
$Q_j \subset \frac 13 \Qz$ for $j \in A_i$) for all $i \in
\mathcal{I}$.  
  For each $i \in \mathcal{I}$ we define local approximation $\bfz_i$ for $\bfz$ on $Q_i$ by
  \begin{align}
    \label{eq:def_zj}
    \bfz_i &:=
    \Pi^0_{I_i}\Pi^1_{B_i}(\bfz),
  \end{align}
  where $\Pi^1_{B_j}(\bfz)$ is the first order averaged Taylor
  polynomial \cite{BreSco94,DieR07} with respect to space and $\Pi^0_{I_i}$ is the zero order
averaged Taylor polynomial in time. We set
\begin{align}
  \label{eq:defwlalt}
  \zal &=
  \begin{cases}
    \bfz & \qquad \text{on $\tfrac 14 \Qz \setminus \Oal$},
    \\
    \sum_{i \in \mathcal{I}} \phi_i \bfz_i &\qquad \text{on $\tfrac 14
      \Qz \cap \Oal$}.
  \end{cases}
\end{align}
We apply the arguments used in the proof of Theorem \cite{BrDS} (Thm. 2.2) to the constant sequence $\bfu_m=\bfu$ with the choice $\alpha=1$. We set
\begin{align}
    \label{eq:def-ulambda}
    \bfu_\lambda:=\curl(\zeta\bfz_\lambda)+\curl(\zeta(\bfw-\bfz)),
  \end{align}
  where $\zeta\in C^\infty_0(\frac 16 Q_0)$ with $\chi_{\frac 18
    Q_0}\leq\zeta\leq \chi_{\frac 16 Q_0}$. This means on $\frac 18
  Q_0$ there holds
  \begin{align*}
    \bfu_\lambda&= \bfu + \curl(\bfz_\lambda-\bfz).
  \end{align*}
We immediately obtain the claim of \ref{itm:1}. As a consequence of
the Lemmas 2.1, 2.4 and 2.9 in \cite{BrDS}
we gain the inequalities 
\begin{align}  \label{eq:z}
\begin{aligned}
  \dashint_{\frac{1}{4} Q_0}
    \abs{\nabla \bfz_\lambda}^{\sigma}\dxt&\leq c\bigg(\dashint_{Q_0}
    \abs{\bfu}^{\sigma}+\dashint_{\frac{1}{3}Q_0} r_0^{\sigma}\abs{ \partial_t\bfz}^{\sigma}\dxt\bigg),\\
  \dashint_{\frac{1}{4}Q_0}\abs{\nabla^2 \bfz_\lambda}^{\sigma}\dxt&\leq c\bigg(\dashint_{Q}r_0^{-\sigma}
    \abs{\bfu}^{\sigma}+\abs{\nabla \bfu}^{\sigma}\dxt+\dashint_{\frac 13 Q_0} \abs{ \partial_t\bfz}^{\sigma}\dxt\bigg),\\
  \dashint_{\frac{1}{4}Q_0}\abs{\partial_t \bfz_\lambda}^{\sigma}\dxt&\leq c\bigg(\dashint_{Q_0}r_0^{-\sigma}
    \abs{\bfu}^{\sigma}+\abs{\nabla \bfu}^{\sigma}\dxt+\dashint_{\frac 13 Q_0} \abs{ \partial_t\bfz}^{\sigma}\dxt\bigg),
\end{aligned}
  \end{align}
which imply the estimates
\begin{align*}  
  \dashint_{Q_0}
    \abs{\bfu_\lambda}^{\sigma}\dxt&\leq c\bigg(\dashint_{Q_0}
    \abs{\bfu}^{\sigma}+\dashint_{\frac{1}{3}Q_0} r_0^{\sigma}\abs{ \partial_t\bfz}^{\sigma}\dxt\bigg),\\
  \dashint_{Q_0}\abs{\nabla \bfu_\lambda}^{\sigma}\dxt&\leq c\bigg(\dashint_{Q_0}r_0^{-\sigma}
    \abs{\bfu}^{\sigma}+\abs{\nabla \bfu}^{\sigma}\dxt+\dashint_{\frac{1}{3}Q_0} \abs{ \partial_t\bfz}^{\sigma}\dxt\bigg).
  \end{align*}
Finally we can replace $\partial_t\bfz$ by $\bfH$ on account of (\ref{eq:z0}) if we replace $\frac{1}{3}Q_0$ by $Q_0$. \\
It remains to find good levels. So we set
  \begin{align*}
    g &:= \mathcal M_s(\chi_{\frac 13 Q_0}\abs{\nabla^2 \bfz})+  \mathcal M_s(\chi_{\frac 13 Q_0}\abs{\partial_t\bfz})
  \end{align*}
  and gain from the continuity of $\mathcal M_s$ and (\ref{eq:z0})
  \begin{align}\label{eq:solparaneu}
\begin{aligned}
  \int_{\R^{d+1}} \abs{g}^\sigma \,\dx &\leq c\bigg(\int_{\frac{1}{3} Q_0}
    \abs{\nabla^2 \bfz}^{\sigma}\dxt+\int_{\frac{1}{3} Q_0} \abs{ \partial_t\bfz}^{\sigma}\dxt\bigg)\\
&\leq c\bigg(\int_{Q_0}
    r_0^{-\sigma}\abs{\bfu}^{\sigma}+\abs{\nabla \bfu}^{\sigma}\dxt+\int_{Q_0} \abs{ \bfH}^{\sigma}\dxt\bigg).
\end{aligned}
  \end{align}
  Furthermore, there holds for every $m_0\in\N$ and every $\gamma>0$
    \begin{align*}
    \int_{\R^{d+1}} \abs{g}^\sigma \,\dx &= \int_{\R^{d+1}} \int_0^\infty
    \sigma t^{\sigma-1} \chi_{\{\abs{g}>t\}} \,\dt \,\dx\\& \geq
    \int_{\R^{d+1}} \sum_{m= m_0}^{2m_0-1} \sigma (2^m\gamma)^{\sigma}
    \chi_{\{\abs{g}>\gamma 2^{m+1}\}} \,\dx.
  \end{align*}
So, there is $m_1\in\set{m_0,...,2m_0-1}$ such that
  \begin{align*}
   \int_{\R^{d+1}} (2^{m_1}\gamma)^{\sigma}
    \chi_{\{\abs{g}>\gamma 2^{m_1+1}\}} \,\dx\leq \frac{c}{m_0} \int_{\R^{d+1}} \abs{g}^\sigma \,\dx .
  \end{align*}
  Setting $\lambda=\gamma 2^{m_1+1}$ we obtain
    \begin{align*}
   \lambda^{\sigma}
    |\tfrac{1}{3} Q_0\cap\set{\abs{g}>\lambda}|\leq \frac{c}{m_0} \int_{\R^{d+1}} \abs{g}^\sigma \,\dx .
  \end{align*}
  Combining this with \eqref{eq:solparaneu} gives the estimate in b) on $\frac{1}{8}Q_0$ due to the definition of $\mathcal O_\lambda$. A simple scaling argument allows us to obtain the desired estimates on $\frac 12 Q_0$.\\ 
  We have $\bfu_\lambda-\bfu=\curl(\bfz_\lambda-\bfz)$ on $\frac{1}{8}Q$ such that (\ref{eq:z0})$_3$ and (\ref{eq:z})$_3$ imply
  $$\partial_t(\bfu_\lambda-\bfu) \in L^{\sigma'}\big(\tfrac 18 I_0,W^{-1,\sigma'}(\tfrac 18 B_0)\big).$$ Moreover, we gain for $\bfphi\in C^\infty_0(\frac{1}{8}Q_0)$ and every $\kappa>0$
  \begin{align*}
   -\int_{\frac 18 Q_0} (\bfu-\bfu_\lambda)\cdot\partial_t\bfphi\dxt=\int_{\frac 18 Q_0} \chi_{\mathcal O_\lambda}\partial_t(\bfz-\bfz_\lambda)\cdot\curl\bfphi\dxt\\
    \leq c(\kappa)\int_{\frac 12 Q_0} \chi_{\mathcal O_\lambda}|\nabla\bfphi|^{\sigma'}\dxt+\kappa \bigg(\dashint_{Q_0}
    \abs{\partial_t(\bfz_\lambda-\bfz)}^{\sigma}\dxt\bigg)
   \end{align*}
   as a consequence of Young's inequality. Applying \eqref{eq:z0}$_3$ and \eqref{eq:z}$_3$ yields
\begin{align*}    
\dashint_{Q_0}
    \abs{\partial_t(\bfz_\lambda-\bfz)}^{\sigma}\dxt
&\leq \,c\,\bigg(\dashint_{Q_0}
    r_B^{-\sigma}\abs{\bfu}^{\sigma}+\abs{\nabla \bfu}^{\sigma}\dxt+\dashint_{Q_0} \abs{ \bfH}^{\sigma}\dxt\bigg).
 \end{align*}  
 So we have shown the estimate claimed in \ref{itm:4} on $\frac{1}{8}Q_0$.
  The same scaling argument as in b) yields the estimate for $\frac 12 Q_0$.
\end{proof}

\newpage

\section{$\mathcal{A}$-Stokes approximation - evolutionary case}
\label{sec:Astokespara}
Let $B\subset\R^d$ be a ball and $J=(t_0,t_1)$ a bounded interval. We set $Q=J\times B$.
For a function $\bfw\in L^1(Q)$ with $\partial_t \bfw \in L^{q'}(J;W^{-1,q'}_{\Div}(B))$ we introduce
the unique function $\bfH_\bfw\in L^{q'}_0(Q)$ with
\begin{align*}
\int_{Q} \bfw\cdot \partial_t\bfvarphi\dxt=\int_Q \bfH_\bfw:\nabla\bfvarphi\dxt
\end{align*}
for all $\bfvarphi\in C^\infty_{0,\Div}(Q)$.
We begin with a variational inequality for the non-stationary $\mathcal{A}$-Stokes
system.

\begin{lemma}
  \label{lem:varineqS} Suppose that \eqref{eq:ell} holds and that $q>1$. There holds for every
$\bfu\in C^0_w([t_0,t_1];L^1(B))\cap L^q(J;W^{1,q}(B))$ with $\bfu(t_0,\cdot)=0$ a.e. 
  \begin{align*}
    \dashint_Q &\abs{\nabla\bfu}^q\dxt
    \leq \,c\, \sup_{\bfxi \in C^\infty_{0,\Div}(Q)} \bigg[\, \dashint_Q\Big(
    \mathcal{A}(\ep(\bfu), \ep(\bfxi))- \bfu\cdot \partial_t\bfxi\Big)\dxt\\&\qquad\qquad\qquad\qquad\qquad\qquad - \dashint_Q \Big(\abs{\nabla
      \bfxi}^{q'}+|\bfH_\bfxi|^{q'}\Big)\dxt\bigg],
  \end{align*}
  where $c$ only depends on~$\mathcal{A}$, $q$ and $d$.
\end{lemma}
\begin{proof}
  Duality arguments show that
  \begin{align*}
    \frac{1}{q}\dashint_Q \abs{\nabla\bfu}^q\dxt = \, \sup_{\bfG \in L^{q'}(Q)}
    \bigg[ \,\dashint_Q \nabla\bfu:\bfG\dxt - \frac{1}{q'}\dashint_Q \abs{
      \bfG}^{q'}\dxt \bigg].
  \end{align*}
  For a given $\bfG\in L^{q'}(Q)$ let $\bfz_\bfG$ be the unique
  $L^{q'}(J;W^{1,q'}_{0,\Div}(B))$-solution to
  \begin{align}\label{eq:13}
  \dashint_{Q} \bfz\cdot \partial_t\bfxi\dxt  + \int_Q \mathcal{A}(\ep(\bfz),\ep(\bfxi)) \dxt = \int_Q \bfG
    : \nabla \bfxi \dxt
  \end{align}
  for all $\bfxi \in C^\infty_{0,\Div}((t_0,t_1]\times B)$. This is a backward parabolic equation with end datum zero. We have that $\partial_t\bfz_\bfG\in L^{q'}(J;W^{-1,q'}_{\Div}(B))$ such that test-functions can be chosen from the space $L^{q}(J;W^{1,q}_{0,\Div}(B))$.
  Due to Theorem \ref{lem:reg2} (which can be applied to $\tilde{\bfz}_{\tilde{\bfG}}(t,\cdot)=\bfz_\bfG(t_1-t,\cdot)$, where $\tilde{\bfG}(t,\cdot)=\bfG(t_1-t,\cdot)$) this solution
  satisfies 
  \begin{align*}
    \dashint_Q|\nabla \bfz_{\bfG}|^{q'}\dxt+\dashint_{Q}|\bfH_{\bfz_\bfG}|^{q'}\dxt\leq\,
    c\,\dashint_Q|\bfG|^{q'}\dxt.
  \end{align*}
  In other words, the mapping $L^{q'}(B)\ni\bfG\mapsto\bfz_\bfG\in
 L^{q'}(J; W^{1,q'}_{0,\Div}(B))$ is continuous. This and $\bfu(t_0,\cdot)=0$
yields (using $\bfu$ as a test-function in \eqref{eq:13})
  \begin{align*}
    \dashint_Q &\abs{\nabla\bfu}^q\dxt\\& \leq \, c\sup_{\bfG \in
      L^{q'}(Q)} \bigg[ \dashint_Q \mathcal{A}(\ep(\bfu),
    \ep(\bfz_\bfG))\dxt-\dashint_{Q} \partial_t\bfz_\bfG\cdot \bfu\dxt \\&\qquad\qquad\qquad\qquad\qquad\qquad- \dashint_Q \Big(\abs{ \nabla\bfz_\bfH}^{q'}+|\bfH_{\bfz_\bfG}|^{q'}\Big)\dxt
    \bigg]
    \\
    &\leq c\, \sup_{\bfxi \in C^\infty_{0,\Div}(Q)} \bigg[ \dashint_Q
    \mathcal{A}(\ep(\bfu), \ep(\bfxi))\dxt-\dashint_{Q} \bfu\cdot \partial_t\bfxi\dxt\\
    &\qquad\qquad\qquad\qquad\qquad\qquad - \dashint_Q \Big(\abs{\nabla
      \bfxi}^{q'}+|\bfH_\bfxi|^{q'}\Big)\dxt\bigg]
  \end{align*}
which yields the claim.
\end{proof} 
Let us now state the $\mathcal{A}$-Stokes approximation. In the following let $B$ be a ball with radius~$r$ and $J$ an interval with length $2r^2$. Let $\tilde{Q}$ denote either
  $Q=J\times B$ or $2Q$. We use similar notations for $\tilde{J}$ and $\tilde{B}$.
\begin{them}
  \label{Astokes} Suppose that \eqref{eq:ell} holds.
 Let $\bfv\in L^{qs}(2\tilde{J}; W^{1,qs}_{\Div}(2\tilde{B}))$, $q,s>1$, be an
  {almost $\mathcal{A}$-Stokes solution}
 in the sense that
  \begin{align}
  \label{eq:almAStokespara}
  \begin{aligned}
    \bigg|\dashint_{2Q} \bfv\cdot \partial_t\bfxi\dxt&-\dashint_{2Q} \mathcal{A}(\ep(\bfv),\ep(\bfxi))\dxt\bigg|\leq
    \delta \dashint_{2\tilde{Q}} |\ep(\bfv)|\dxt\,\|\nabla\bfxi\|_\infty
    \end{aligned}
  \end{align}
  for all $\bfxi\in C^{\infty}_{0,\Div}(2Q)$ and some
  small~$\delta>0$. Then the unique solution $\bfw \in
  L^q(J;W^{1,q}_{0,\Div}(B))$ to
  \begin{align}
    \label{eq:vH}
    \begin{aligned}
   \int_Q \bfw\cdot \partial_t\bfxi\dxt  &-\int_Q \mathcal{A}(\ep(\bfw),\ep(\bfxi)) \dxt\\
   &= \int_Q \bfv\cdot \partial_t\bfxi\dxt -
    \int_Q \mathcal{A}(\ep(\bfv),\ep(\bfxi)) \dxt
    \end{aligned}
  \end{align}
  for all $\bfxi \in C^\infty_{0,\Div}([t_0,t_1)\times B)$ satisfies
  \begin{align*}
\dashint_{Q} \Big|\frac{\bfw}{r}\Big|^q\dxt+\dashint_{Q}
    |\nabla\bfw|^q\dxt\leq \kappa\bigg(\dashint_{2\tilde{Q}}
    |\nabla\bfv|^{qs}\dxt\bigg)^{\frac{1}{s}}.
  \end{align*}
  It holds $\kappa=\kappa(q,s,\delta)$ and
  $\lim_{\delta\rightarrow0}\kappa(q,s,\delta)=0$. The function $\bfh
  := \bfv - \bfw$ is called the $\mathcal{A}$-Stokes approximation of
  $\bfv$. 
\end{them}
\begin{remark}
From the proof of Theorem \ref{Astokes} we gain the following stability result choosing $p=q s=q$.
\begin{align*}
 \dashint_{Q} \Big|\frac{\bfw}{r}\Big|^p\dxt+\dashint_{Q}
    |\nabla\bfw|^p\dxt\leq c\,\dashint_{2\tilde{Q}}
    |\nabla\bfv|^{p}\dxt.
  \end{align*}
  Indeed $\kappa$ stays bounded if $s\rightarrow 1$.
\end{remark}

\begin{proof}
  Let $\bfw$ be defined as in~\eqref{eq:vH}.  Combining Poincar{\'e}'s inequality with Lemma \ref{lem:varineqS}
  and~~\eqref{eq:vH} shows
  \begin{align}
    \label{eq:Liptrunvib}
    \begin{aligned}
   \dashint_{Q} &\Big|\frac{\bfw}{r}\Big|^q\dxt+\dashint_{Q}
      |\nabla\bfw|^q\dxt\\& \leq c\, \sup_{\bfxi \in
        C^\infty_{0,\Div}(Q)} \bigg[ \dashint_Q \mathcal{A}(\ep(\bfv),
      \ep(\bfxi))\dxt-\dashint_Q \bfv\cdot \partial_t\bfxi\dxt \\
    &\qquad\qquad\qquad\qquad\qquad\qquad- \dashint_Q \Big(\abs{\nabla \bfxi}^{q'}+|\bfH_\bfxi|^{q'}\Big)\dxt\bigg].
      \end{aligned}
  \end{align}
  In the following let us fix $\bfxi \in C^\infty_{0,\Div}(Q)$. Let
  $$\gamma:= \bigg(\,\dashint_Q \abs{\nabla \bfxi}^{q'}\dxt+\dashint_Q \abs{ \bfH_\bfxi}^{q'}\dxt\bigg)^{\frac 1{q'}}$$
  and $m_0 \in \setN$, $m_0\gg1$. Due to Theorem~\ref{cor:appl} applied
  with $\sigma=q'$ we find $\lambda \in [2^{m_0}\gamma, 2^{2m_0} \gamma]$ and
  $\bfxi_\lambda \in L^\infty(4J;W^{1,\infty}_{0,\Div}(4B))$ such that
  \begin{align}
    \label{eq:xi1}
    \norm{\nabla \bfxi_\lambda}_{L^\infty(4Q)} &\leq c\, \lambda,
    \\
    \label{eq:xi2}
    \lambda^{q'} \frac{\mathcal{L}^{d+1}\left(2Q\cap\set{\bfxi_\lambda \not=
          \bfxi}\right)}{\abs{Q}} &\leq \frac{c}{m_0} \bigg(\dashint_Q
    \abs{\nabla \bfxi}^{q'}\dxt+\dashint_Q\abs{ \bfH_\bfxi}^{q'}\dxt\bigg),
    \\
    \label{eq:xi3}
    \dashint_{4Q} \abs{\bfxi_\lambda}^{q'} \dxt &\leq c
    \bigg(\dashint_Q \abs{\bfxi}^{q'} \dxt+\dashint_Q r^{q'}\abs{ \bfH_\bfxi}^{q'}\dxt\bigg),\\
       \label{eq:xi4}
    \dashint_{4Q} \abs{\nabla \bfxi_\lambda}^{q'} \dxt &\leq c
    \bigg(\dashint_Q \abs{\nabla \bfxi}^{q'} \dxt+\dashint_Q \abs{ \bfH_\bfxi}^{q'}\dxt\bigg).
    \end{align}
  Note that $\bfxi$ can be extended by 0 to $4Q$ thus the equation
  \begin{align*}
  \partial_t\bfxi=\Div\Bog_B(\partial_t\bfxi)=:\Div\bfH_{\bfxi}
  \end{align*}
  holds on $4Q$ by the properties of $\Bog_B$ (since $\bfH_\bfxi$ can be extended as well).  Theorem~\ref{cor:appl} \ref{itm:4}
      implies that $\partial_t(\bfxi-\bfxi_\lambda)\in L^{q'}(2J,W^{-1,q'}(2B))$ and
      \begin{align}\label{eq:xi5}
      \begin{aligned}
   \int_{2J} &\langle\partial_t(\bfxi-\bfxi_\lambda),\bfphi\rangle\dt\\
&   \leq c(\kappa)\int_{2Q} \chi_{\set{\bfxi\neq\bfxi_\lambda}}|\nabla\bfphi|^{q}\dxt+\kappa \bigg(\int_{Q}\abs{\nabla \bfxi}^{q'}+ \abs{ \bfH_\bfxi}^{q'}\dxt\bigg)
\end{aligned}
   \end{align}
   for all $\bfphi\in W^{1,q}_0(2Q)$.
  We calculate for $\eta\in C^\infty_0(2Q)$ with $\eta\equiv1$ on $Q$, $|\nabla^k\eta|\leq cr^{-k}$ and $|\partial_t\nabla^{k-1}\eta|\leq cr^{-(k+1)}$ ($k=1,2$)
  \begin{align*}
    \dashint_Q &\mathcal{A}(\ep(\bfv), \ep(\bfxi))\dxt -\dashint_{Q} \bfv\cdot \partial_t\bfxi\dxt \\
    &=2^{d+2}\dashint_{2Q} \mathcal{A}\big(\ep(\bfv), \ep(\eta\bfxi-\Bog_{2B\setminus B}(\nabla\eta\bfxi))\big)\dxt -\dashint_{2Q} \bfv\cdot \partial_t\big(\eta\bfxi-...\big)\dxt \\
     &=2^{d+2}\bigg(\dashint_{2Q} \mathcal{A}\big(\ep(\bfv), \ep(\eta\bfxi_\lambda-\Bog_{2B\setminus B}(\nabla\eta\bfxi_\lambda))\big)\dxt +\dashint_{2Q} \partial_t\bfv\cdot \big(\eta\bfxi_\lambda-...\big)\dxt \bigg)\\
 &+2^{d+2}\dashint_{2Q} \mathcal{A}\big(\ep(\bfv), \ep(\eta(\bfxi-\bfxi_\lambda)-\Bog_{2B\setminus B}(\nabla\eta(\bfxi-\bfxi_\lambda)))\big)\dxt\\
 &+2^{d+2}\dashint_{2Q} \partial_t\bfv\cdot \big(\eta(\bfxi-\bfxi_\lambda)-\Bog_{2B\setminus B}(\nabla\eta(\bfxi-\bfxi_\lambda))\big)\dxt\\
    &=: 2^{d+2}(I + II+III).
  \end{align*}
  Note that the time-derivative of $\bfv$ exists in the $W^{-1,\infty}_{\Div}$-sense as a consequence of (\ref{eq:almAStokespara}). Therefore all terms are well-defined by the properties of $\bfxi_\lambda$.
  We have the following inequality on account of the continuity properties of $\nabla\Bog$ on $L^p$-spaces, \eqref{eq:xi3},~\eqref{eq:xi4} and Poincar\'e's  inequality (we set $\tilde{\bfxi_\lambda}:=\bfxi-\bfxi_\lambda$):
  \begin{align}\label{eq:anew}
  \begin{aligned}
 \dashint_{2Q}|\nabla\bfPsi_\lambda|^{q'}\dxt &:=\dashint_{2Q}
    \abs{\nabla(\eta\tilde{\bfxi}_\lambda)-\nabla\Bog_{2B\setminus B}( \nabla\eta\tilde{\bfxi}_\lambda)}^{q'}\dxt\\&\leq c\,\dashint_{2Q}
    \abs{\nabla \tilde{\bfxi}_\lambda}^{q'}\dxt+c\, \dashint_{2Q}
    \Big|\frac{ \tilde{\bfxi}_\lambda}{r}\Big|^{q'}\dxt\\
    &\leq  c\,\dashint_{Q}
    \abs{\nabla \bfxi}^{q'}\dxt+c\,\dashint_{Q}
    \Big|\frac{ \bfxi}{r}\Big|^{q'}\dxt+c\,\dashint_{Q}
    \abs{\bfH_\bfxi}^{q'}\dxt\\
     &\leq  c\,\dashint_{Q}
    \abs{\nabla \bfxi}^{q'}\dxt+c\,\dashint_{Q}
    \abs{\bfH_\bfxi}^{q'}\dxt.
    \end{aligned}
  \end{align}
 Young's inequality for an appropriate choice of $\varepsilon>0$ implies together with (\ref{eq:xi3}) and (\ref{eq:xi4})
  \begin{align*}
    II
       &\leq c(\varepsilon)\, \dashint_{2Q} \abs{\ep(\bfv)}^q \chi_{\set{\bfxi \not=
        \bfxi_\lambda}} \dxt + \varepsilon\, \dashint_{2Q}
    \abs{\nabla \bfPsi_\lambda}^{q'}\dxt\\
          &\leq c\, \dashint_{2Q} \abs{\ep(\bfv)}^q \chi_{\set{\bfxi \not=
        \bfxi_\lambda}} \dxt + \frac{1}{3}\, \dashint_{Q}
    \abs{\nabla \bfxi}^{q'}+\abs{\bfH_\bfxi}^{q'}\dxt\\
    &=: II_{1} + II_{2},
  \end{align*}
  where $c$ depends on $\mathcal{A}$, $q$ and
  $q'$. With H{\"o}lder's inequality we gain
  \begin{align*}
    II_{1} &\leq c \bigg(\dashint_{2Q}  \abs{\nabla \bfv}^{qs}
    \dxt\bigg)^{\frac 1s} \bigg( \frac{\mathcal{L}^{d+1}\left(2Q\cap\set{\bfxi_\lambda
          \not= \bfxi}\right)}{\abs{Q}} \bigg)^{1 - \frac 1s}.
  \end{align*}
  If follows from \eqref{eq:xi2}, by the choice of~$\gamma$ and
  $\lambda \geq \gamma$ that
  \begin{align}\label{eq:kappam0}
    \frac{\mathcal{L}^{d+1}\left(2Q\cap\set{\bfxi_\lambda \not=
          \bfxi}\right)}{\abs{Q}} &\leq \frac{c \gamma^{q'}}{m_0
      \lambda^{q'}} \leq \frac{c}{m_0}.
  \end{align}
  Thus
  \begin{align*}
    II_{1} &\leq c \bigg(\dashint_{2Q} \abs{\nabla
      \bfv}^{qs}\dxt\bigg)^{\frac 1s} \bigg( \frac{c}{m_0} \bigg)^{1 -
      \frac 1s}.
  \end{align*}
  We choose $m_0$ so large such that
  \begin{align*}
    II_{1} &\leq \frac{\kappa}{3} \bigg(\dashint_{2Q} 
    \abs{\nabla \bfv}^{qs}\dxt\bigg)^{\frac 1s}.
  \end{align*}

  Since $\partial_t(\bfxi-\bfxi_\lambda)\in L^{q'}(2J,W^{-1,q'}(2B))$ we can write $III$ as
  \begin{align*}
  III&=\dashint_{2Q} \bfv\cdot \partial_t\eta(\bfxi-\bfxi_\lambda)\dxt+\dashint_{2Q} \eta\bfv\cdot \partial_t(\bfxi-\bfxi_\lambda)\dxt\\
  &-\dashint_{2Q} \bfv\cdot \Bog_{2B\setminus B}(\partial_t\nabla\eta(\bfxi-\bfxi_\lambda))\dxt\\
  &-\dashint_{2Q} \Bog_{2B\setminus B}^*(\bfv)\nabla\eta\cdot\partial_t(\bfxi-\bfxi_\lambda)\dxt\\
  &=:III_1+III_2+III_3+III_4.
  \end{align*}
The \Bogovskii-operator is continuous from $L^2_0\rightarrow L^2$. Hence
its dual (in the sense of $L^2$-duality) is continuous from $L^2\rightarrow L^2_0$.
     Therefore $\Bog_{2B\setminus B}^*(\bfv)$ is well-defined. We consider the four terms separately and obtain
    for the first one
  \begin{align*}
  III_1&\leq \,c\,\dashint_{2I}\dashint_{2B\setminus B} \chi_{\set{\bfxi_\lambda\neq\bfxi}}\Big|\frac{\bfv}{r}\Big|\Big|\frac{\bfxi-\bfxi_\lambda}{r}\Big|\dxt\\
  &\leq c(\varepsilon)\dashint_{2I}\dashint_{2B\setminus B} \Big|\frac{\bfv}{r}\Big|^q\chi_{\set{\bfxi_\lambda\neq\bfxi}}\dxt+\varepsilon \dashint_{2Q} \Big|\frac{\bfxi-\bfxi_\lambda}{r}\Big|^{q'}\dxt\\
&=:c(\varepsilon)III_{11}+\varepsilon III_{12}.
  \end{align*}
Poincar\'{e}'s inequality and Young's inequality yield
    \begin{align*}
    III_{11} &\leq c \bigg(\dashint_{2I}\dashint_{2B\setminus B}  \Big|\frac{\bfv}{r}\Big|^{qs}
    \dxt\bigg)^{\frac 1s} \bigg( \frac{\mathcal{L}^{d+1}\left(2Q\cap\set{\bfxi_\lambda
          \not= \bfxi}\right)}{\abs{Q}} \bigg)^{1 - \frac 1s}\\
          &\leq c \bigg(\dashint_{2Q}  \abs{\nabla\bfv}^{qs}
    \dxt\bigg)^{\frac 1s} \bigg( \frac{\mathcal{L}^{d+1}\left(2Q\cap\set{\bfxi_\lambda
          \not= \bfxi}\right)}{\abs{Q}} \bigg)^{1 - \frac 1s}
  \end{align*}
  Arguing as done for the term $II_1$ implies
    \begin{align*}
    III_{11} &\leq \frac{\kappa}{12} \bigg(\dashint_{2Q} 
    \abs{\nabla \bfv}^{qs}\dxt\bigg)^{\frac 1s}.
  \end{align*}
  Moreover, we gain from (\ref{eq:xi3}) and Poincar\'{e}'s inequality
  \begin{align*}
  III_{12}&\leq c\,\dashint_{Q} \Big|\frac{\bfxi}{r}\Big|^{q'}\dxt+c\,\dashint_{Q}|\bfH_\bfxi|^{q'}\dxt\\
  &\leq c\,\dashint_{Q} |\nabla\bfxi|^{q'}\dxt+c\,\dashint_{Q}|\bfH_\bfxi|^{q'}\dxt
  \end{align*}
  and finally
    \begin{align*}
  III_1
  &\leq \frac{\kappa}{12} \bigg(\dashint_{2Q} 
    \abs{\nabla \bfv}^{qs}\dxt\bigg)^{\frac 1s}+\frac{1}{12}\bigg(\int_{Q}\abs{\nabla \bfxi}^{q'}+ \abs{ \bfH_\bfxi}^{q'}\dxt\bigg).
   \end{align*}
    The formulation in (\ref{eq:almAStokespara}) does not change if we subtract terms which are constant in space from $\bfv$ (note that $(\partial_t\bfxi)=0$ for every $t$ due to $\partial_t\bfxi(t,\cdot)\in C^\infty_{0,\Div}(B)$). So we can assume 
\begin{align}\label{eq:meanv}    
    \dashint_{2B\setminus B}\bfv(t)\dx=0\quad\text{for a.e.}\quad t\in 2J.
\end{align}    
  As a consequence of (\ref{eq:xi5}), (\ref{eq:meanv}) and Poincar\'{e}'s inequality we obtain similarly as for $III_1$
  \begin{align*}
  III_2&\leq c(\varepsilon)\int_{2Q} \chi_{\set{\bfxi\neq\bfxi_\lambda}}|\nabla(\eta\bfv)|^{q}\dxt+\varepsilon\bigg(\int_{Q}\abs{\nabla \bfxi}^{q'}+ \abs{ \bfH_\bfxi}^{q'}\dxt\bigg)\\
  &\leq \frac{\kappa}{12} \bigg(\dashint_{2Q} 
    \abs{\nabla \bfv}^{qs}\dxt\bigg)^{\frac 1s}+\frac{1}{12}\bigg(\int_{Q}\abs{\nabla \bfxi}^{q'}+ \abs{ \bfH_\bfxi}^{q'}\dxt\bigg)
   \end{align*}
  Taking into account continuity properties of the \Bogovskii -operator from $L^q_0\rightarrow W^{1,q}_0$ we can estimate $III_3$ via (we use again (\ref{eq:meanv}) and Poincar\'{e}'s inequality)
  \begin{align*}
  III_{3}&\leq \bigg(\dashint_{2Q}\dashint_{2B\setminus B}\Big|\frac{\bfv}{r}\Big|^{qs}\dxt\bigg)^{\frac{1}{qs}}\bigg(\dashint_{2Q}r^{(qs)'}\Big|\Bog_{2B\setminus B}(\partial_t\nabla\eta(\bfxi-\bfxi_\lambda))\big)\Big|^{(qs)'}\dxt\bigg)^{\frac{1}{(qs)'}}\\
  &\leq \,c\bigg(\dashint_{2Q}|\nabla\bfv|^{qs}\dxt\bigg)^{\frac{1}{qs}}\bigg(\,\dashint_{2Q}r^{2(qs)'}\big|\partial_t\nabla\eta(\bfxi-\bfxi_\lambda)\big|^{(qs)'}\dxt\bigg)^{\frac{1}{(qs)'}}\\
   &\leq \,c\bigg(\dashint_{2Q}|\nabla\bfv|^{qs}\dxt\bigg)^{\frac{1}{qs}}\bigg(\,\dashint_{2Q}\chi_{\set{\bfxi\neq\bfxi_\lambda}}\Big|\frac{\bfxi-\bfxi_\lambda}{r}\Big|^{(qs)'}\dxt\bigg)^{\frac{1}{(qs)'}}.
  \end{align*}
We gain from Young's inequality for every $\varepsilon>0$
    \begin{align*}
  III_{3}
   &\leq \,\frac{\varepsilon^{q}}{12}\bigg(\dashint_{2Q}|\nabla\bfv|^{qs}\dxt\bigg)^{\frac 1s}+c\,\varepsilon^{-q'}\bigg(\dashint_{2Q}\chi_{\set{\bfxi\neq\bfxi_\lambda}}\Big|\frac{\bfxi-\bfxi_\lambda}{r}\Big|^{(qs)'}\dxt\bigg)^{\frac{q'}{(qs)'}}\\
   &=:\frac{\varepsilon^{q}}{12}III_{31}+c\varepsilon^{-q'}III_{32}.
  \end{align*} 
   It holds due to H\"older's inequality, Poincar\'{e}'s inequality, \eqref{eq:xi2}, \eqref{eq:xi4} and \eqref{eq:kappam0} for $m_0$ large enough
   \begin{align*}
   III_{32}&\leq \bigg( \frac{\mathcal{L}^{d+1}\left(2Q\cap\set{\bfxi_\lambda
          \not= \bfxi}\right)}{\abs{Q}} \bigg)^{1 - \frac 1s}\bigg(\dashint_{2Q}\Big|\frac{\bfxi-\bfxi_\lambda}{r}\Big|^{q'}\dxt\bigg)\\
&\leq\,c\, \bigg( \frac{\mathcal{L}^{d+1}\left(2Q\cap\set{\bfxi_\lambda
          \not= \bfxi}\right)}{\abs{Q}} \bigg)^{1 - \frac 1s}\bigg(\dashint_{4Q}\Big|\frac{\bfxi-\bfxi_\lambda}{4r}\Big|^{q'}\dxt\bigg)\\
&\leq \,c\,\bigg( \frac{\mathcal{L}^{d+1}\left(2Q\cap\set{\bfxi_\lambda
          \not= \bfxi}\right)}{\abs{Q}} \bigg)^{1 - \frac 1s}\bigg(\dashint_{4Q}\big(|\nabla\bfxi|^{q'}+|\nabla\bfxi_\lambda|^q\big)\dxt\bigg)\\
          &\leq \frac{\kappa}{12 c}\bigg(\int_{Q}\abs{\nabla \bfxi}^{q'}+ \abs{ \bfH_\bfxi}^{q'}\dxt\bigg).
   \end{align*}
   Choosing $\varepsilon:=\kappa^{1/q'}$ implies
     \begin{align*}
  III_3
  &\leq \frac{\kappa}{12 } \bigg(\dashint_{2Q} 
    \abs{\nabla \bfv}^{qs}\dxt\bigg)^{\frac 1s}+\frac{1}{12}\bigg(\int_{Q}\abs{\nabla \bfxi}^{q'}+ \abs{ \bfH_\bfxi}^{q'}\dxt\bigg).
   \end{align*}
 By (\ref{eq:xi5}) and (\ref{eq:kappam0}) we have for $m_0$ large enough
        \begin{align*}
  III_4
  &\leq  \,c\,\dashint_{2Q}\chi_{\set{\bfxi\neq\bfxi_\lambda}}
    \abs{\nabla(\nabla\eta\Bog_{2B\setminus B}^*(\bfv))}^{q}\dxt+\frac{1}{12}\bigg(\int_{Q}\abs{\nabla \bfxi}^{q'}+ \abs{ \bfH_\bfxi}^{q'}\dxt\bigg)\\
     &\leq  \,\varepsilon\,\bigg(\dashint_{2Q}\abs{\nabla(\nabla\eta\Bog_{2B\setminus B}^*(\bfv))}^{sq}\dxt\bigg)^{\frac{1}{s}}+\frac{1}{12}\bigg(\int_{Q}\abs{\nabla \bfxi}^{q'}+ \abs{ \bfH_\bfxi}^{q'}\dxt\bigg)\\
    &=:\varepsilon III_{41}+\frac{1}{12}III_{42}.
   \end{align*}
Due to the continuity of $\Bog(\Div(\cdot))$ on $L^p$ for any $1<p<\infty$ (see \cite{Ga1}[III.3, Theorem 3.3] and \cite[Thm. 3.7]{BrCi} for the \Bogovskii-operator and negative norms) we have continuity of $\nabla \Bog^\ast$ as well. 
This, Poincar\'e's inequality (note that $\Bog_{2B\setminus B}^*(\bfv)\in L^p_0(2B\setminus B)$) and \eqref{eq:meanv} yield
   \begin{align*}
   III_{41}&\leq \,c\,\bigg(\dashint_{2Q}\Big|\frac{\Bog_{2B\setminus B}^*(\bfv)}{r^2}\Big|^{sq}\dxt+\dashint_{2Q}\Big|\frac{\nabla\Bog_{2B\setminus B}^*(\bfv)}{r}\Big|^{sq}\dxt\bigg)^{\frac{1}{s}}\\
   &\leq \,c\,\bigg(\dashint_{2Q}\Big|\frac{\nabla\Bog_{2B\setminus B}^*(\bfv)}{r}\Big|^{sq}\dxt\bigg)^{\frac{1}{s}}
    \leq \,c\,\bigg(\dashint_{2I}\dashint_{2B\setminus B}\Big|\frac{\bfv}{r}\Big|^{sq}\dxt\bigg)^{\frac{1}{s}}\\
    &\leq \,c\,\bigg(\dashint_{2Q}|\nabla \bfv|^{sq}\dxt\bigg)^{\frac{1}{s}}
   \end{align*}
   and hence for $\varepsilon:=\kappa/12 c$
    \begin{align*}
  III_4
  &\leq \frac{\kappa}{12} \bigg(\dashint_{2Q} 
    \abs{\nabla \bfv}^{qs}\dxt\bigg)^{\frac 1s}+\frac{1}{12}\bigg(\int_{Q}\abs{\nabla \bfxi}^{q'}+ \abs{ \bfH_\bfxi}^{q'}\dxt\bigg).
   \end{align*}
 Plugging the estimates for $III_1$-$III_4$ together we see
    \begin{align*}
  III
  &\leq \frac{\kappa}{3} \bigg(\dashint_{2Q} 
    \abs{\nabla \bfv}^{qs}\dxt\bigg)^{\frac 1s}+\frac{1}{3}\bigg(\int_{Q}\abs{\nabla \bfxi}^{q'}+ \abs{ \bfH_\bfxi}^{q'}\dxt\bigg).
   \end{align*}
  Since $\bfv$ is an almost $\mathcal{A}$-Stokes solution and
  $\norm{\nabla\bfxi_\lambda}_{\infty} \leq c\, \lambda \leq c\, 2^{m_0}
  \gamma$ we have
  \begin{align*}
    |I| &\leq \delta\,\dashint_{2 \tilde{Q}} \abs{\nabla \bfv}\dxt\,
    \|\nabla \bfxi_\lambda\|_{\infty,2Q}\\
    & \leq \delta\,\bigg( \dashint_{2 \tilde{Q}} \abs{\nabla \bfv}^{qs}\dxt\bigg)^{\frac{1}{qs}}
  \, c\,2^{m_0} \gamma.
  \end{align*}
  We apply Young's inequality and Jensen's inequality to gain
  \begin{align*}
    |I| &\leq  \delta 2^{m_0}  c\bigg(\dashint_{2 \tilde{Q}} 
    \abs{\nabla \bfv}^{qs}\dxt \bigg)^{\frac 1s}+\delta 2^{m_0}  c\gamma^{q'}
    \\
    &\leq \delta 2^{m_0}  c\bigg(\dashint_{2 \tilde{Q}} 
    \abs{\nabla \bfv}^{qs}\dxt \bigg)^{\frac 1s} + \delta 2^{m_0}c
    \bigg(\dashint_{Q} \abs{\nabla \bfxi}^{q'}\dxt+\dashint_{Q} \abs{\bfH_\bfxi}^{q'}\dxt\bigg).
  \end{align*}
  Now, we choose $\delta>0$ so small such that $\delta 2^{m_0} c\leq
  \kappa/3$. Thus
  \begin{align*}
    |I| &\leq \frac{\kappa}{3} \bigg(\dashint_{2 \tilde{Q}}
    \abs{\nabla \bfv}^{qs}\dxt \bigg)^{\frac 1s} +
    \frac{1}{3}\bigg(\dashint_{Q} \abs{\nabla \bfxi}^{q'}\dxt+\dashint_{Q} \abs{\bfH_\bfxi}^{q'}\dxt\bigg).
  \end{align*}
  Combining the estimates for $I$, $II$ and $III$ we have established
  \begin{align*}
    \dashint_{2Q} \!\mathcal{A}&(\ep(\bfv) ,\ep(\bfxi))\dxt
     -\dashint_{Q} \bfv\cdot \partial_t\bfxi\dxt\\&\leq \kappa \bigg(\dashint_{2 \tilde{Q}} \abs{\nabla
      \bfv}^{qs}\dxt\bigg)^{\frac 1s} + \dashint_{Q}
    \abs{\nabla \bfxi}^{q'}\dxt+\dashint_{Q} \abs{\bfH_\bfxi}^{q'}\dxt.
  \end{align*}
Inserting this in~\eqref{eq:Liptrunvib} shows the claim
\end{proof}

\subsection*{Acknowledgement}
\begin{itemize}
\item The work of the author was partially supported by Leopoldina (German National Academy of Science).
\item The author wishes to thank R. Farwig for very helpful advices regarding the $L^q$-theory for the Stokes system.
\end{itemize}

\end{document}